\documentclass[11pt,a4paper]{amsart}
\usepackage[margin=20mm]{geometry}
\usepackage[numbers]{natbib}
\usepackage{hyperref}

\usepackage{amsaddr}

\usepackage{amssymb,amsmath}
\usepackage{amsthm}
\usepackage{bbm}
\usepackage{stmaryrd}
\usepackage[usenames,dvipsnames]{xcolor}
\usepackage{color}
\usepackage{enumitem}

\input xy
\xyoption{all} %\tolerance=500

\numberwithin{equation}{section}

\newtheorem{theorem}{Theorem}[section]
\newtheorem{corollary}[theorem]{Corollary}

\newtheorem{lemma}[theorem]{Lemma}
\newtheorem{proposition}[theorem]{Proposition}

\theoremstyle{definition}
\newtheorem{definition}[theorem]{Definition}
\newtheorem{remark}[theorem]{Remark}
\newtheorem{example}[theorem]{Example}

\newcommand{\rmd}{\textnormal{d}}

\DeclareMathOperator{\Vect}{Vect}
\DeclareMathOperator{\Rank}{Rank}

\DeclareMathOperator{\Span}{Span}

\DeclareMathOperator{\ad}{ad}

\DeclareMathOperator{\Image}{Im}

\DeclareMathOperator{\tr}{tr}
\DeclareMathOperator{\diag}{diag}

\newcommand{\catname}[1]{\textnormal{\texttt{#1}}}

\font\black=cmbx10 \font\sblack=cmbx7 \font\ssblack=cmbx5 \font\blackital=cmmib10  \skewchar\blackital='177
\font\sblackital=cmmib7 \skewchar\sblackital='177 \font\ssblackital=cmmib5 \skewchar\ssblackital='177
\font\sanss=cmss10 \font\ssanss=cmss8 %scaled 900
\font\sssanss=cmss8 scaled 600 \font\blackboard=msbm10 \font\sblackboard=msbm7 \font\ssblackboard=msbm5
\font\caligr=eusm10 \font\scaligr=eusm7 \font\sscaligr=eusm5  \font\fraktur=eufm10
\font\sfraktur=eufm7 \font\ssfraktur=eufm5 
\font\bsymb=cmsy10 scaled\magstep2
\def\all#1{\setbox0=\hbox{\lower1.5pt\hbox{\bsymb
       \char"38}}\setbox1=\hbox{$_{#1}$} \box0\lower2pt\box1\;}
\def\exi#1{\setbox0=\hbox{\lower1.5pt\hbox{\bsymb \char"39}}
       \setbox1=\hbox{$_{#1}$} \box0\lower2pt\box1\;}

\def\tx#1{{\fam0\relax#1}}

\newfam\bifam
\textfont\bifam=\blackital \scriptfont\bifam=\sblackital \scriptscriptfont\bifam=\ssblackital

\newfam\blfam
\textfont\blfam=\black \scriptfont\blfam=\sblack \scriptscriptfont\blfam=\ssblack

\newfam\bbfam
\textfont\bbfam=\blackboard \scriptfont\bbfam=\sblackboard \scriptscriptfont\bbfam=\ssblackboard

\newfam\ssfam
\textfont\ssfam=\sanss \scriptfont\ssfam=\ssanss \scriptscriptfont\ssfam=\sssanss
\def\sss#1{{\fam\ssfam\relax#1}}

\newfam\clfam
\textfont\clfam=\caligr \scriptfont\clfam=\scaligr \scriptscriptfont\clfam=\sscaligr

\newfam\frfam
\textfont\frfam=\fraktur \scriptfont\frfam=\sfraktur \scriptscriptfont\frfam=\ssfraktur

\def\hpb#1{\setbox0=\hbox{${#1}$}
    \copy0 \kern-\wd0 \kern.2pt \box0}
\def\vpb#1{\setbox0=\hbox{${#1}$}
    \copy0 \kern-\wd0 \raise.08pt \box0}

\def\pmb#1{\setbox0\hbox{${#1}$} \copy0 \kern-\wd0 \kern.2pt \box0}
\def\pmbb#1{\setbox0\hbox{${#1}$} \copy0 \kern-\wd0
      \kern.2pt \copy0 \kern-\wd0 \kern.2pt \box0}
\def\pmbbb#1{\setbox0\hbox{${#1}$} \copy0 \kern-\wd0
      \kern.2pt \copy0 \kern-\wd0 \kern.2pt
    \copy0 \kern-\wd0 \kern.2pt \box0}
\def\pmxb#1{\setbox0\hbox{${#1}$} \copy0 \kern-\wd0
      \kern.2pt \copy0 \kern-\wd0 \kern.2pt
      \copy0 \kern-\wd0 \kern.2pt \copy0 \kern-\wd0 \kern.2pt \box0}
\def\pmxbb#1{\setbox0\hbox{${#1}$} \copy0 \kern-\wd0 \kern.2pt
      \copy0 \kern-\wd0 \kern.2pt
      \copy0 \kern-\wd0 \kern.2pt \copy0 \kern-\wd0 \kern.2pt
      \copy0 \kern-\wd0 \kern.2pt \box0}

\mathchardef\za="710B  %\alpha
\mathchardef\zb="710C  %\beta
\mathchardef\zg="710D  %\gamma
\mathchardef\zd="710E  %\delta
\mathchardef\zve="710F %\epsilon
\mathchardef\zz="7110  %\zeta
\mathchardef\zh="7111  %\eta
\mathchardef\zvy="7112 %\theta
\mathchardef\zi="7113  %\iota
\mathchardef\zk="7114  %\kappa
\mathchardef\zl="7115  %\lambda
\mathchardef\zm="7116  %\mu
\mathchardef\zn="7117  %\nu
\mathchardef\zx="7118  %\xi
\mathchardef\zp="7119  %\pi
\mathchardef\zr="711A  %\rho
\mathchardef\zs="711B  %\sigma
\mathchardef\zt="711C  %\tau
\mathchardef\zu="711D  %\upsilon
\mathchardef\zvf="711E %\phi
\mathchardef\zq="711F  %\chi
\mathchardef\zc="7120  %\psi
\mathchardef\zw="7121  %\omega
\mathchardef\ze="7122  %\varepsilon
\mathchardef\zy="7123  %\vartheta
\mathchardef\zf="7124  %\varomega
\mathchardef\zvr="7125 %\varrho
\mathchardef\zvs="7126 %\varsigma
\mathchardef\zf="7127  %\varphi
\mathchardef\zG="7000  %\Gamma
\mathchardef\zD="7001  %\Delta
\mathchardef\zY="7002  %\Theta
\mathchardef\zL="7003  %\Lambda
\mathchardef\zX="7004  %\Xi
\mathchardef\zP="7005  %\Pi
\mathchardef\zS="7006  %\Sigma
\mathchardef\zU="7007  %\Upsilon
\mathchardef\zF="7008  %\Phi
\mathchardef\zW="700A  %\Omega
\mathchardef\zC="7009  %\Psi

\newcommand{\be}{\begin{equation}}
\newcommand{\ee}{\end{equation}}

\newcommand{\bea}{\begin{eqnarray}}
\newcommand{\eea}{\end{eqnarray}}
\def\*{{\textstyle *}}
\newcommand{\R}{{\mathbb R}}

\newcommand{\s}{{\textstyle *}}

\def\Sec{\sss{Sec}}
\def\Vect{\sss{Vect}}

\def\sT{{\sss T}}

\def\sF{{\sss F}}

\def\xi{\tx{i}}

\def\s*{{\scriptstyle *}}

\def\cO{\mathcal{O}}

\newcommand{\beas}{\begin{eqnarray*}}
\newcommand{\eeas}{\end{eqnarray*}}
\title{Carrollian Lie Algebroids: Taming Singular Carrollian Geometries}
\author{Andrew James Bruce}   
   \email{andrewjamesbruce@googlemail.com}
  
   \date{\today}
 
\begin{document}
\begin{abstract}
Developments in Carrollian gravity and holography necessitate the use of singular Carroll vector fields, a feature that cannot be accommodated within standard Carrollian geometry. We introduce  Carrollian Lie algebroids as a framework to study such singular Carrollian geometries. In this approach, we define the Carroll distribution as the image of the kernel of the degenerate metric under the anchor map, i.e., the map from the Lie algebroid to the tangent bundle of the manifold. The Carroll distribution is, in general, a singular Stefan--Sussmann distribution that will fluctuate between rank-1 and rank-0, and so captures the notion of a singular Carroll vector field. As an example, we show that an invariant Carrollian structure on a principal bundle leads to a Carrollian structure on the associated Atiyah algebroid that will, in general, have a singular Carroll distribution. Mixed null-spacelike hypersurfaces, under some simplifying assumptions, also lead to examples of Carrollian Lie algebroids.  Furthermore, we establish the existence of compatible connections on  Carrollian Lie algebroids, and as a direct consequence, we conclude that Carrollian manifolds can always be equipped with compatible affine connections.  \par
\end{abstract}

 \maketitle

\setcounter{tocdepth}{2}
 \tableofcontents
 
\begin{flushright}
\emph{It would be so nice if something made sense for a change. }\\
Lewis Carroll, Alice’s Adventures in Wonderland, (1865) 
\end{flushright}

\section{Introduction} 
The search for a consistent quantum theory of gravity has been ongoing for approximately a century. Indeed, Einstein pointed out in 1916 that quantum effects would modify general relativity, and in 1927, Klein proposed that quantum theory would radically change the nature of spacetime.  Despite a long history, see Rovelli \cite{Rovelli:2002} for an overview, there is currently no fully-fledged theory of quantum gravity. Over recent decades, there has been an increasing interest in non-Lorentzian geometries (see Figueroa-O’Farrill \cite{Figueroa-O’Farrill:2022}, for example), largely inspired by their potential to give insight into quantum gravity. Carrollian geometry\footnote{Jean-Marc Lévy-Leblond \cite{Lévy-Leblond:1965} coined the term Carrollian because in the limit $c \mapsto 0$, massive particles move at infinite speed while remaining stationary. This is somewhat similar to the situation described in the dialogue between the Red Queen and Alice in Lewis Carroll’s \emph{Through the Looking Glass}. Also see \cite{Lévy-Leblond:2023}.}, which can be phrased  in terms of the Carrollian limit $c \rightarrow 0$ of Lorentzian geometries, has become a significant area of research within theoretical physics; the mathematical literature is far less developed.  It is quite remarkable that geometries in which the light cones become lines in the time direction have much to do with physics at all.  In particular, two events at spatially distinct points are causally disconnected, and the traditional notions of past and future events become severely limited.  Carrollian structures offer a novel perspective on limiting situations in general relativity and possibly offer a simplified window into quantum gravity.  Furthermore,  Carrollian holography posits a duality between a (non-Carrollian) gravitational theory in a $(d+1)$-dimensional asymptotically flat spacetime and a Carrollian field theory on its $d$-dimensional null boundary  (see \cite{Alday:2025,Donnay:2022,Donnay:2023,Lipstein:2025} and references there in).  Fontanella  \&  Payne \cite{Fontanella:2025} argue that the correct bulk theory here might in fact be a Carrollian string theory.  Carrollian physics may also offer insight into dark energy and inflation (see de Boer et al. \cite{deBoer:2022}). We also point out that Carrollian dynamics has found applications in hydrodynamics and condensed matter physics (see \cite[Part III]{Bagchi:2025} for an overview). \par 
The initial papers on the Carrollian group, understood as the $c\rightarrow 0$ limit of the Poincaré group, were published by Lévy-Leblond \cite{Lévy-Leblond:1965} and Sen Gupta \cite{SenGupta:1966} in the mid-1960s. Heanneaux \cite{Heanneaux:1979} was the first to consider intrinsic Carrollian geometry. Duval et al. (see \cite{Duval:2014,Duval:2014a,Duval:2014b}) extended this to intrinsic conformal Carrollian geometry, which has sparked recent interest in Carrollian physics.  Recall that a Carrollian manifold is a manifold equipped with a degenerate metric whose kernel is spanned by a nowhere vanishing complete vector field. For a review of many of the facets of Carrollian physics, the reader may consult Bagchi et al. \cite{Bagchi:2025}. Natural examples of Carrollian manifolds include null hypersurfaces, such as punctured future or past light-cones in Minkowski spacetime, and the event horizon of a Schwarzschild black hole.\par 
Ecker et al. \cite{Ecker:2023}, study the analogue of black holes in Carrollian gravity (specifically, within $2d$ Carroll–Jackiw–Teitelboim gravity). As there is no light-cone structure in Carrollian gravity, the notion of an event horizon becomes meaningless. Instead, one must consider points where the Carroll field either diverges or vanishes  (see \cite[Section 2.3.4]{Ecker:2023}). The necessity of a singular Carroll vector field in gravitational physics also appears in Donnay et al. \cite{Donnay:2022}. Such singular vector fields, i.e., Carroll vector fields that have a non-empty locus where they vanish, are not included in the standard definition of a Carrollian manifold.  In particular, Ecker et al. pose the question of developing a mathematical framework for handling singular Carroll vector fields while the rest of the geometry remains intact. That is, is there a consistent framework in which the basic tenets of Carrollian geometry remain unchanged, while the Carroll vector field may be singular? These singular points would, at least in part, represent a species of Carrollian singularities, a notion that is under development.  We also point out that divergent Carroll vector fields, i.e., one or more components of the vector field become infinite, also feature as Carrollian singularities; however, we will not address this situation here.\par 
The solution we give is to work with \emph{Carrollian Lie algebroids}, a notion we will define shortly. Recall that a Lie algebroid\footnote{We remark that Lie algebroids have been applied to various aspects of gravitational physics, examples include \cite{Anastasiei:2014,Barnich:2017,Blohmann:2013,Bruce:2016,Fabi:2014,Strobl:2004,Vacaru:2012}.} is a generalisation of the tangent bundle of a manifold and a Lie algebra. Lie algebroids provide a broader setting for formulating differential geometry, and crucially for this paper, they are vital in the study of singular distributions.  A Carrollian structure on a Lie algebroid $\pi :A\rightarrow M$ is the direct analogue of a Carrollian structure on a manifold. That is, we have a degenerate metric $g$ on a Lie algebroid $A$ and a rank-1 Lie subalgebroid $L \subset A$, such that $\ker(g) = \Sec(L)$. The potentially singular nature of the geometry is encoded in the anchor map $\rho : A \rightarrow \sT M$. Specifically, we define the \emph{Carroll distribution} $\mathcal{C} := \rho(L) \subset \sT M$, which is, in general, a singular Stefan--Sussmann distribution as it may fluctuate between rank-1 and rank-0.  Note that the rank of the kernel of the degenerate metric remains constant. We interpret the Carroll distribution as defining a null direction at the non-singular points on the base manifold. At the singular points, the null direction `collapses' and the Carrollian nature of the base manifold is lost.  We comment that standard Carrollian manifolds are examples of Carrollian Lie algebroids, demonstrating that our framework is a generalisation of the standard theory. As a physically relevant example, we construct, under some simplifying assumptions, Carrollian Lie algebroids associated with mixed null-spacelike hypersurfaces in a Lorentzian manifold. \par 
The potential downside of this formalism is the abstraction; the fundamental structure is defined on a Lie algebroid over the manifold under study rather than on the manifold itself. This makes the physical interpretation of the degenerate metric and the rank-1 Lie subalgebroid unclear. We liken the situation to the BV--BRST formalism, where the space of fields is expanded to include ghosts and anti-fields, ensuring that the theory is mathematically consistent. Nonetheless, Carrollian Lie algebroids provide a rigorous  mathematical structure to understand singular Carrollian geometries. \par
Alongside establishing the foundations of Carrollian Lie algebroids, we establish the existence of Lie algebroid connections that are compatible with the Carrollian structure. A direct consequence of the general existence results is that compatible connections can always be constructed on any (weak) Carrollian manifold.  Thus, we have the foundational aspects of singular Carrollian geometry in place to further explore Carrollian physics. For example, Carrollian Lie algebroids and their compatible connections may be  useful in further studies of extreme gravitational phenomena in the Carrollian limit and Carrollian holography.\par 
To aid the reader less familiar with the formalism of Lie algebroids, we provide a dictionary translating the standard constructions in Carrollian geometry to the setting of Carrollian Lie algebroids.

\medskip
\renewcommand{\arraystretch}{1.5}
\noindent \begin{tabular}[h]{ |p{230pt}||p{230pt} | }
 \hline
 \multicolumn{2}{|c|}{\textbf{Dictionary}} \\
 \hline
 Carrollian Lie Algebroids & Carrollian Manifolds \\
 \hline
 Vector bundle $A$ over a manifold $M$  & Tangent bundle $\sT M$ of a manifold $M$   \\
 Lie bracket of sections $[u,v]$ & Lie bracket of vector fields $[X,Y]$ \\
 Anchor map $\rho : A \rightarrow \sT M$ & Identity map $\rho: \sT M \rightarrow \sT M$\\
Leibniz rule $[u, fv] = \rho_u(f)\, v + f\,[u,v]$   & Leibniz rule $[X, fY]= X(f)Y + f\, [X,Y]$\\
 Metric $g : \Sec(A) \times \Sec(A) \rightarrow C^\infty(M)$&  Metric $g :\Vect(M)\times \Vect(M) \rightarrow C^\infty(M)$   \\
 Trivial line subbundle $L := \ker(g)\subset A$& Trivial line subbundle $L := \ker(g) \subset \sT M$  \\
 Nowhere vanishing section $\sigma \in \Sec(L)$& Carroll vector field $\kappa \in \Sec(L)$ \\
 Possibly singular Carroll distribution\newline  $\mathcal{C} := \rho(L) \subset \sT M $& Regular  distribution\newline  $\mathcal{C} := L \subset \sT M $\\
 Possibly singular Carroll foliation \newline i.e, the dimension of each  leaf may be  $0$ or $1$  & Regular Carrollian foliation\newline i.e, the dimension of each leaf is $1$ \\
 \hline
\end{tabular}

\medskip

\noindent \textbf{Mathematical Motivation.}  The \emph{Lie algebroid mantra} states that whatever works on a tangent bundle works on a Lie algebroid.  The point of view here is that Lie algebroids are a mixture of a tangent bundle and a Lie algebra; both are limiting cases of a Lie algebroid, and they offer a wider framework than tangent bundles and related tensor bundles to formulate concepts in differential geometry. For example, Riemannian Lie algebroids have a long history (see \cite{Boucetta:2011} and references therein).   In particular, the core tenets of Riemannian geometry generalise to the setting of Lie algebroids directly. The case of degenerate metrics on Lie algebroids is far less studied, though we remark that they appear in the context of information geometry on Lie groupoids (see Grabowska et al. \cite{Grabowska:2020}). Thus, it is  natural  from the perspective of abstract differential geometry to investigate the core properties of Carrollian Lie algebroids as part of a wider effort to understand degenerate metrics on Lie algebroids.
\medskip

\noindent \textbf{Arrangement.} The bulk of this paper is in Section \ref{sec:CarrollStructures}. We recall the notion of a Lie algebroid as well as degenerate metrics and connections on them in Subsection \ref{subsec:LieAlgebroids}. In Subsection \ref{subsec:CarrollianLieAlgebroids}, we define Carrollian Lie algebroids (Definition \ref{def:CarrollianLieAlgebroid}) and make an initial study of them. A list of foundational examples of Carrollian Lie algebroids is given in Subsection \ref{subsec:Examples}. In Subsection \ref{subsec:MixHypSur}, we show that, under some simplifying assumptions, that mixed null-spacelike hypersurfaces can be described using Carrollian Lie algebroids.  Compatible connections on a Carrollian Lie algebroid are the subject of Subsection \ref{subsec:Connections}. We end on Section \ref{sec:ConRem} with concluding remarks and suggestions for future directions.

\medskip 
\noindent \textbf{Conventions.} All manifolds will be finite-dimensional, second-countable, Hausdorff, and smooth.  We will almost exclusively work in the global coordinate-free approach to differential geometry. 
%
%%%%%%%%%%%%%%%%%%%%%%%
%
\section{Carrollian Structures on Lie algebroids}\label{sec:CarrollStructures}
\subsection{Recollection of Lie algebroids}\label{subsec:LieAlgebroids}
We assume that the reader has a grasp of the fundamental theory of Lie algebroids. Our general reference for Lie algebroids is the book by Mackenzie \cite{Mackenzie:2005}.  Lie algebroids appear in diverse contexts, including  singular distributions, Poisson geometry, geometric mechanics, control theory, and gauge theory.  Our main point of view is that Lie algebroids are  ``generalised tangent bundles''. For completeness, we will give the definition of a Lie algebroid.
\begin{definition}[Pradines \cite{Pradines:1974}]\label{def:LieAlg}
A \emph{Lie algebroid} is a triple $(A, [-,-], \rho)$, where $\pi : A \rightarrow M$ is a vector bundle, $[-,-]$ is a Lie bracket on the vector space $\Sec(A)$, referred to as the \emph{Lie algebroid bracket}, and  a linear map $\rho: \Sec(A) \rightarrow \Vect(M)$, referred to as the \emph{anchor}, that satisfies the Leibniz rule
$$[u,f\,v] = \rho_u(f) \, v  +  \, f \, [u,v],$$
with $u$ and $v \in \Sec(A)$ and $f \in C^\infty(M)$. 
\end{definition}
Note that the anchor is a homomorphism of Lie algebras, i.e., $\rho_{[u,v]} = [\rho_u, \rho_v]$. We will also refer to the map $\rho : A \rightarrow \sT M$ as the anchor; the context will make it clear.\par 
By a \emph{distribution}, we mean the assignment of a vector subspace of the tangent space at each point on a manifold. A distribution is \emph{regular} if the dimension of each assigned vector subspace is all equal; otherwise, we have a \emph{singular} or  \emph{Stefan--Sussmann distribution}. A distribution is said to be \emph{integrable} if it is generated by a family of vector fields, and is invariant under the flow of every vector field generating the distribution.  An integrable distribution leads to foliation of the manifold, i.e., a  partition of the manifold into submanifolds. Note that for Stefan--Sussmann distributions, these submanifolds will have a range of dimensions.\par
The image of the anchor $\Image(\rho) \subseteq \sT M$ is, in general, an integrable Stefan--Sussmann distribution.  A Lie algebroid is said to be a \emph{regular Lie algebroid} if  the anchor map $\rho : A\rightarrow \sT M$ is of constant rank. In this case, the image of the anchor is an integrable regular distribution. A Lie algebroid is said to be \emph{transitive} if the anchor map is a surjection, i.e., is onto.  It can easily be shown that the Lie bracket of any Lie algebroid is local, and that a Lie algebroid can be restricted to a Lie algebroid over any open subset of the base manifold.
\begin{example}
Some fundamental examples of Lie algebroids include 
\begin{enumerate}[itemsep=0.5em]
\item the tangent bundle of a manifold $\sT M \rightarrow M$ equipped with the Lie bracket of vector fields, and the identity map as the anchor; 
\item any integrable distribution $\mathcal{D}\subset \sT M$ where the bracket and anchor are inherited from the tangent bundle;
\item any vector bundle $E \rightarrow M$ equipped with the zero bracket and anchor; 
\item bundles of Lie algebras, where the Lie bracket is defined point-wise and the anchor is the zero map. Specifically, Lie algebras are Lie algebroids over a single point;
\item the cotangent bundle $\sT^*M \rightarrow M$ of a Poisson manifold $(M, \Lambda)$ where the anchor is the contraction of one-form $\Sec(\sT^* M)\ni \alpha \mapsto \Lambda(\alpha, - ) \in \Vect(M)$, and  the Lie bracket is  $[\alpha, \beta] = \mathcal{L}_{\Lambda(\alpha ,-)}\beta  - \mathcal{L}_{\Lambda(\beta ,-)}\alpha - \rmd\big(\Lambda(\alpha, \beta)\big)  $.
\end{enumerate}
For further examples, the reader may consult Mackenzie \cite{Mackenzie:2005}.
\end{example}
Metrics are, of course, vital in differential geometry and physics; the notion directly generalises to Lie algebroids. A non-degenerate metric on a Lie algebroid is just a fibrewise inner product on the underlying vector bundle; there are no extra compatibility conditions. We thus make the following definition.
\begin{definition}
A (possibly degenerate) \emph{metric $g$ on a Lie algebroid} $(A, [-,-], \rho)$ is a symmetric $C^\infty(M)$-bilinear  map
$$g : \Sec(A) \times \Sec(A) \longrightarrow C^\infty(M)\,.$$
A metric on a Lie algebroid is said to be \emph{non-degenerate} if the kernel,
$$\ker(g) := \{u \in \Sec(E)~~|~~ g(u,v) =0, ~\textnormal{for all}~v\in \Sec(E) \}\,,$$
is trivial, i.e., consists of just the zero section, and is said to be \emph{degenerate} otherwise. A degenerate metric $g$ is said to be a \emph{regular metric} if $\ker(g)$ is of constant rank.
\end{definition}
A metric on a Lie algebroid provides each fibre with a semi-inner product, i.e., an ``inner product'' that may be degenerate. Hence, we will refer to a Lie algebroid equipped with a metric as a \emph{semi-Riemannian Lie algebroid}.  Moreover, associated with a metric is the \emph{flat map}
\begin{align}
    \flat : ~ & \Sec(A) \longrightarrow \Sec(A^*)\\
    \nonumber & u \mapsto u^\flat := g(u,-)\,,
\end{align}
which is an isomorphism only in the case that $g$ is non-degenerate.
\begin{example}
A manifold  equipped with a (possibly degenerate) metric $(M,g)$ gives rise to a semi-Riemannian Lie algebroid structure on $\sT M$.  As a specific example,  consider $M = \R^4$ equipped with Cartesian coordinates $(w,x, y , z )$, and the metric given as matrix
$$g = \begin{pmatrix}
    0 & 0 & 0 & 0 \\
    0 & 0 & 0& 0 \\
    0& 0 & 1& 0 \\
    0 & 0 & 0& 1 
  \end{pmatrix}\,,$$
  with respect to the coordinate basis $(\partial_w, \partial_x, \partial_y, \partial_z)$. Observe that $\det(g) =0$, and $g$ is a degenerate metric. Clearly, $\ker(g) = \Span \big\{  \partial_w, \partial_x\big\}$, and thus this example is a regular degenerate metric. 
\end{example}
\begin{example}
Quadratic Lie algebras are Lie algebras equipped with a non-degenerate symmetric bilinear form that is invariant under the adjoint action. A little more explicitly $g : \mathfrak{g} \otimes \mathfrak{g} \rightarrow \R$  can be viewed as a metric, and we further have the invariant condition $g([u,v],w) + g(v, [u,w]) =0$, for all $u,v,w \in \mathfrak{g}$. Semisimple Lie algebras are canonical examples - the Killing form provides the required metric. Thus, quadratic Lie algebras form a class of Riemannian Lie algebroids. 
\end{example}
\begin{example}
Any Lie algebroid $(A, [-,-], \rho)$ can be equipped with the zero metric, which is defined as $g(u,v) := 0$ for all $u,v \in \Sec(A)$.  As $\ker(g) = \Sec(A)$, we have a regular degenerate metric whose kernel is of maximal rank.
\end{example}
The \emph{Lie derivative of $g$} with respect to a section $u \in \Sec(A)$ is defined via the Leibniz rule with respect to contraction, and  is thus given by
\begin{equation}
\big(\mathcal{L}_u g\big)(v,w) := \rho_u(g(v,w)) - g ([u,v], w) - g(v,[u,w])\,,
\end{equation}
for all $v,w \in \Sec(A)$. It can be shown that $\mathcal{L}_{[u,v]} = [\mathcal{L}_u, \mathcal{L}_v]$.
\begin{definition}
Let $(A, [-,-] , \rho, g)$ be a semi-Riemannian Lie algebroid. A section $u \in \Sec(A)$ is said to be a \emph{Killing section} if $\mathcal{L}_u g =0$.  \end{definition}
Clearly, the set of Killing sections of a semi-Riemannian Lie algebroid forms a Lie subalgebra of $\big(\Sec(A), [-,-]\big)$. We remark that the Lie algebra of Killing sections will, in general, be infinite-dimensional when the metric is degenerate.  For a degenerate metric, the inverse does not exist, which implies that the number of independent Killing equations is reduced.  In particular, the components of a Killing section in a `null direction' are unconstrained and can be arbitrary smooth functions. Thus, the Lie algebra of Killing sections will be infinite. \par 
Connections, in their various guises, are widespread throughout differential geometry and mathematical physics. The notion of a connection on a Lie algebroid as a generalisation of a Koszul connection,  i.e., a covariant derivative of sections of a vector bundle, is direct. Importantly, we also have the notions of curvature and torsion tensors. 
\begin{definition}
 Let $(A, [-,-], \rho)$ be a Lie algebroid. A \emph{Lie algebroid connection} is an $\R$-bilinear map that satisfies 
 \begin{enumerate}[itemsep=0.5em]
\item $ \nabla_{f u}v = f \, \nabla_u v$; and
\item $\nabla_u (fv) = f \, \nabla_u v + \rho_u(f) v\,,$
\end{enumerate}
for all $u,v \in \Sec(A)$ and $f \in C^\infty(M)$. The \emph{curvature tensor} and \emph{torsion tensor} of a Lie algebroid connection are defined, respectively, as
\begin{align*}
    R(u,v)w := [\nabla_u, \nabla_v]w - \nabla_{[u,v]} w\,, &&\textnormal{and}&& T(u,v) := \nabla_u v - \nabla_v u  - [u,v]\,, 
\end{align*}
for all $u,v, w \in \Sec(A)$.
\end{definition}
Due to the skewsymmetry of the Lie algebroid bracket, it is clear that 
$$R(u,v) = - R(v,u)\,, \qquad \textnormal{and}\qquad  T(u,v) = - T(u,v)\,.$$
Via following the standard arguments, it can be shown that any Lie algebroid connection satisfies the \emph{algebraic and differential Bianchi identities}
\begin{subequations}
\begin{align}
   & \sum_{\mathrm{cyclic}(u,v,w)} \left( R(u,v)w - (\nabla_u T)(v,w) - T\big(T(u,v)w \big)\right) =0\,,\\ 
   & \sum_{\mathrm{cyclic}(u,v,w)} \left( (\nabla_u R)(v,w) + R\big( T(u,v),w \big)\right)  = 0\,,
\end{align}
\end{subequations}
for all $u,v,w \in \Sec(A)$.\par 
For any pair of sections $u,v \in \Sec(L)$, we have the $C^\infty(M)$-linear map defined by the curvature tensor of a Lie algebroid connection
\begin{equation}\label{eqn:RLinOp}
R(u, v) : \Sec(A) \longrightarrow \Sec(A)\,.
\end{equation}
Via modification of the standard arguments for the existence of affine connections, the existence of Lie algebroid connections can be established, see  Křižka \cite{Křižka:2008}, for example.
\begin{example}
A Lie algebroid connection on the tangent bundle  $\sT M \rightarrow M$  is an affine connection on $M$.  The curvature and torsion tensors are the standard tensors in differential geometry. If $(M,g)$ is a (pseudo-)Riemannian manifold, then the affine connection may be chosen to be the Levi–Civita connection.
\end{example}
\begin{example}
 Any vector bundle $\pi :A \rightarrow M$ equipped with the zero anchor and zero Lie bracket is a Lie algebroid. In this case, a Lie algebroid connection is a $C^\infty(M)$-bilinear map
 $$\nabla : \Sec(A) \times \Sec(A) \rightarrow \Sec(A)\,.$$
 The curvature and torsion tensors simplify to
 $$R(u,v)w := [\nabla_u, \nabla_v]w \,, \qquad  T(u,v) := \nabla_u v - \nabla_v u $$\,.
\end{example}
\begin{example}
A connection on a (real) Lie algebra $(\mathfrak{g}, [-,-])$  is an $\R$-bilinear map $\nabla : \mathfrak{g} \times \mathfrak{g} \rightarrow \mathfrak{g}$.    
\end{example}
\begin{definition}
A connection on a semi-Riemannian Lie algebroid $(A, [-,-], \rho, g)$ is said to be \emph{metric compatible} if $\nabla g =0$, i.e.,
$$\rho_u\big( g(v,w)\big) = g(\nabla_u v, w) + g(v, \nabla_u w)\,,$$
for all $u,v,w \in \Sec(A)$.
\end{definition}
\begin{example}
If $(M,g)$ is a semi-Riemannian manifold, then $\sT M$ is a semi-Riemannian Lie algebroid with the anchor being the identity map, and the Lie bracket is the standard Lie bracket of vector fields. In this case, the condition for an affine connection to be metric compatible is precisely the standard notion, i.e., 
$$X \big(g(Y,Z) \big) = g(\nabla_X Y, Z) + g(Y, \nabla_X Z)\,,$$
for all vector fields $X,Y$ and $Z \in \Vect(M)$.
\end{example}
The \emph{Fundamental Theorem of Riemannian Lie Algebroids} states there is a unique connection on a Riemannian Lie algebroid  characterised by the two properties that it has vanishing torsion and is metric compatible (see, for example, Boucetta \cite{Boucetta:2011}). Such connections are known as Levi–Civita connections. This is, of course, the generalisation of the classical theorem of Riemannian manifolds. It should be noted that this theorem does not hold if the metric is degenerate. We will address this carefully for the specific case of Carrollian Lie algebroids in Subsection \ref{subsec:Connections}. 
%
%%%%%%%%%%%%%%%%%%%%%%%
%
\subsection{Carrollian Lie Algebroids}\label{subsec:CarrollianLieAlgebroids}
Recall the definition of a weak Carrollian manifold.
\begin{definition}[Duval et al. \cite{Duval:2014}]
A \emph{weak Carrollian manifold} is a triple $(M, g, \kappa)$, where $M$ is a smooth manifold, $g$ is a degenerate metric whose pointwise diagonalisation is $\diag(1,1,\cdots, 1,0)$, $\kappa$ is a nowhere vanishing complete vector field, called the \emph{Carroll vector field}, that satisfy the compatibility condition  $\ker(g) = \Span\{\kappa\}$.
\end{definition}
 Recall that a complete vector field is a vector field for which every integral curve exists for all `time', i.e., the parameter describing the curve is defined for all $t \in (- \infty, + \infty)$. It is possible to relax the condition on the signature of the degenerate metric and  consider \emph{pseudo-Carrollian manifolds} (see Gibbons \cite{Gibbons:2019}). We will change focus slightly and concentrate on the line bundle associated with the Carroll vector field $\kappa$, and relax the diagonalisation requirement, and make the following definition.
\begin{definition}\label{def:CarrollianLieAlgebroid}
A \emph{Carrollian Lie algebroid} is a quintuple $(A, [-,-], \rho, g, L)$, where $(A, [-,-], \rho)$ is a Lie algebroid over $M$, equipped with a degenerate metric $g$, and $L$ is a trivial line bundle over $M$, such that
\begin{enumerate}[itemsep=0.5em]
\item $L \subset A$ as vector bundles; and
\item $ \ker(g) := \{u \in \Sec(A) ~~|~~ g(u,-) =0 \} = \Sec(L)$.
\end{enumerate}
\end{definition}
\begin{remark}
In the definition of a Carrollian Lie algebroid, we take the kernel of the metric to be a trivial line bundle and so admits a nowhere vanishing section; this implies that, once $A$ is oriented, the Euler class $e(A)$ vanishes. This is consistent with the standard notion of a (weak) Carrollian manifold where the Carroll vector field is nowhere vanishing, i.e., the Carroll vector field provides a frame for a trivial line bundle. The physical reasoning is that the Carroll vector field locally provides a `null direction' and for this to be well-defined everywhere, the vector field must be non-vanishing. However, mathematically, we can relax the triviality condition on the line bundle; this will result in minor changes to the following statements and their proofs. 
\end{remark}
\begin{proposition}\label{prop:LAlgebroid}
Let $(A, [-,-], \rho, g, L)$ be a Carrollian Lie algebroid. The trivial line bundle $\tau : L \rightarrow M$ is a Lie subalgebroid of $A$.
\end{proposition}
\begin{proof}
By definition, $L$ is a vector subbundle of $A$.  We need to show that the $\Sec(L)$  is closed under the (restriction of the) Lie bracket of $A$. This is almost immediate as $L$ is rank-$1$. Explicitly, picking a (global) frame $\sigma \in \Sec(L)$, we set $v_1 = f_1 \sigma$ and $v_2 = f_2 \sigma$, where $f_1$ and $f_2$ are smooth functions on $M$. Then, using the Leibniz rule
$$[v_1, v_2] = [f_1 \sigma , f_2 \sigma] = \big(f_1 \rho_\sigma(f_2) - f_2 \rho_\sigma(f_1) \big) \, \sigma + f_1 f_2 \, [\sigma, \sigma]\,.$$
Then due to skew-symmetry we have $[\sigma, \sigma] =0$, and so 
$$[v_1, v_2] = \big(f_1 \rho_\sigma(f_2) - f_2 \rho_\sigma(f_1) \big) \, \sigma \in \Sec(L)\,.$$
\end{proof}
\begin{proposition}
Let $(A, [-,-], \rho, g, L)$ be a Carrollian Lie algebroid, and let $E := A /L$ be the quotient vector bundle. Then the induced metric $g_E$ on $E$ is well-defined and non-degenerate.
\end{proposition}
\begin{proof}
 The metric $g_E$ is constructed as follows. Let us pick two sections $\bar{u}, \bar{v} \in \Sec(E)$ and lift these to sections $u,v \in \Sec(A)$, i.e.,  we have two specified sections $u$ and $v$ such that $\pi(u) = \bar{u}$ and $\pi(v) = \bar{v}$, where $\pi : A \rightarrow E = A/L$  is the projection. We then define $g_E(\bar{u}, \bar{v}) := g(u,v)$. To show that this is well-defined, we need to demonstrate that the definition does not depend on the chosen lifts. With this in mind, we consider $u' := u +\sigma_1$ and $v' := v + \sigma_2$ for arbitrary $\sigma_1, \sigma_2 \in \Sec(L)$.  We need to show that $g(u', v') = g(u,v)$ for the induced metric to be well-defined. Explicitly, 
 $$g(u', v') = g(u + \sigma_1 , v + \sigma_2) = g(u,v) + g(u, \sigma_2) + g(\sigma_1, v) + g(\sigma_1, \sigma_2)\,.$$
 As $\sigma_1, \sigma_2 \in \ker(g)$ by definition we have that $g(u', v') = g(u,v)$.\par
 To show non-degeneracy, we need to demonstrate that if $\bar{u} \in \Sec(E)$ is such that $g|_E(\bar{u}, \bar{v}) =0$  for all $\bar{v}\in \Sec(E)$, then $\bar{u}$ must be the zero section.  Let us assume  $g|_E(\bar{u}, \bar{v}) =0$ for all $\bar{v}\in \Sec(E)$. This means that $g(u,v) =0$, for all lifts $v \in \Sec(A)$. However, as the kernel of $g$  is $\Sec(L)$, it must be the case that $u \in \Sec(L)$. As $\bar{u}$ is the projection of $u$, it follows that $\bar{u}$ is the zero section of the vector bundle $E$. Thus, $g_E$ is non-degenerate.
\end{proof}
\begin{remark}
The quotient bundle $E := A/L$ is a vector bundle of rank one less than $A$; however, in general, it is not a Lie subalgebroid of $A$. Thus, in general, $(E, g_E)$ is a Euclidean vector bundle.   Moreover, we have a short exact sequence of vector bundles 
$$0 \rightarrow L \rightarrow A \rightarrow E \rightarrow 0\,.$$
\end{remark}
\begin{proposition}
Let $(A, [-,-], \rho, g, L)$ be a Carrollian Lie algebroid. If any nowhere vanishing section of $L$ is Killing, then all sections of $L$ are Killing.
\end{proposition}
\begin{proof}
A nowhere vanishing section $\sigma$ provides a frame for $\Sec(L)$. Let us assume that this section is Killing, i.e., $\mathcal{L}_\sigma g =0$, which is equivalent to
$$\rho_\sigma(g(v,w)) = g([\sigma, v], w)+ g(v, [\sigma, w])\,,$$
for all $v,w\in \Sec(A)$.  Next, consider an arbitrary section $f \sigma$ of $L$ written in this frame, and let us check the Killing condition.  Explicitly using the Leibniz rule and the bi-linearity of the degenerate metric, we have
\begin{align*}
g([f\sigma, v], w)+ g(v, [f\sigma, w]) &= g(-\rho_v(f)\sigma +  f [\sigma, v],w) + g(v, - \rho_w(f)\sigma + f[\sigma, w])\\
&= f \big( g([\sigma, v], w) + g(v, [\sigma, w])\big) - \big( \rho_v(f)g(\sigma, w)+ \rho_w(f)g(v,\sigma)\big)\\
& = f \big( g([\sigma, v], w) + g(v, [\sigma, w])\big)\\
& = f\big(\rho_\sigma(g(v,w)) \big) =  \rho_{f\sigma}(g(v,w))\,,
\end{align*}
as $\sigma$ is in the kernel of $g$, i.e., $g(\sigma, w) = 0$ and $g(v,\sigma) =0$, and $\rho_{f \sigma } = f \, \rho_\sigma$. Thus, $f \sigma$ is a Killing section.
\end{proof}
\begin{definition}\label{def:StatFra}
  A Carrollian Lie algebroid is said to be a
  \begin{enumerate}
      \item \emph{stationary Carrollian Lie algebroid} if every section of $L$ is Killing; 
      \item a \emph{framed Carrollian Lie algebroid} if a frame of $\Sec(L)$ has been fixed.
  \end{enumerate}
\end{definition}
Moving on to morphisms, we have the following.
\begin{definition}
A {morphism of Carrollian Lie algebroids} 
$$\phi :  (A_1, [-,-]_1, \rho_1, g_1, L_1) \longrightarrow (A_2, [-,-]_2, \rho_2, g_2, L_2)\,, $$ 
over the same base manifold $M$, is a vector bundle homomorphism (over the identity) that 
\begin{enumerate}[itemsep=0.5em]
\item is a  Lie algebra homomorphism, i.e., $\phi^\# ([u,v]_1) = [\phi^\#(u), \phi^\#(v)]_2$ and $\rho_2 \circ \phi^\# = \rho_1$;
\item is an isometry, i.e., $g_2(\phi^\#(u), \phi^\#(v)) =  g_1(u,v)$; and
\item  respects the kernels, i.e., $\phi^\#(\Sec(L_1)) = \Sec(L_2)$,
\end{enumerate}
where $\phi^\# : \Sec(A_1) \rightarrow \Sec(A_2)$ is the induced map between sections. \emph{Morphisms of framed Carrollian Lie algebroids} are morphisms of Carrollian Lie algebroids that satisfy the further condition $\phi^\#(\sigma_1) = \sigma_2$, where $\sigma_1$ and $\sigma_2$ are the fixed frames of $\Sec(L_1)$ and $\Sec(L_2)$, respectively. The resulting category of Carrollian Lie algebroids over $M$ we denote as $\catname{CarLieAlg}_M$. Similarly, the category of framed Carrollian Lie algebroids over $M$ we denote as $\catname{FCarLieAlg}_M$.
\end{definition}
The notion of infinitesimal symmetries of a Carrollian Lie algebroid are defined following the classical notion of infinitesimal symmetries of a (weak) Carrollian manifold.
\begin{definition}
Let $(A, [-,-], \rho, g, L)$ be a Carrollian Lie algebroid. A section $u \in \Sec(A)$ is an \emph{infinitesimal symmetry} if 
\begin{enumerate}[itemsep=0.5em]
\item it is Killing, i.e., $\mathcal{L}_u g =0$; and
\item it is an infinitesimal symmetry of $L$, i.e., $\mathcal{L}_u v := [u,v] \in \Sec(L)$, for all $v \in \Sec(L)$. 
\end{enumerate}
\end{definition}
Directly from the properties of the Lie derivative, it can easily be seen that the set of infinitesimal symmetries of a Carrollian Lie algebroid forms a Lie algebra, which we denote as $\mathfrak{carr}(A)$. Associated with this, we define $\rho\big(\mathfrak{carr}(A) \big) =:\mathfrak{carr}(M) \simeq \mathfrak{carr}(A)/\ker(\rho) \subset \Vect(M)$ (where we have used the First Isomorphism Theorem of Lie algebras). We interpret the Lie algebra $\mathfrak{carr}(M)$ as describing infinitesimal diffeomorphisms of $M$ inherited from the larger Lie algebra $\mathfrak{carr}(A)$. As the anchor has, in general, a non-trivial kernel,  there are ``hidden symmetries'' in that there are symmetries of the Carrollian  Lie algebroid that are not `seen' at the level of the base manifold. We remark that, in general, both $\mathfrak{carr}(A)$ and $\mathfrak{carr}(M)$ are infinite-dimensional.  
\begin{definition}
The \emph{Carroll distribution} of a Carrollian Lie algebroid $(A, [-,-], \rho, L)$ is the Stefan--Sussmann distribution   $\mathcal{C} := \rho(L) \subset \sT M$.  The associated singular foliation we refer to as the \emph{Carroll foliation}.
\end{definition}
A Carroll distribution is, in general, a singular (involutive) Stefan--Sussmann distribution that can jump between rank-1 and rank-0. We interpret the Carroll distribution as locally defining a `null direction' in $M$; however, this direction collapses at points where the rank is zero.  We stress that the singular behaviour is in $\mathcal{C}$ while $L$ is a genuine line bundle over $M$. That is, the possible singular nature is due to the anchor. Via the Rank-Nullity Theorem, insisting that $\ker(\rho|_L)$ is trivial, i.e., contains only the zero section, implies that $\mathcal{C}\subset \sT M$ is a non-singular distribution of rank-1. That is, we have a trivial line bundle inside $\sT M$.  In the other extreme, if $\rho|_L$ is the zero map, then the distribution collapses, i.e., is constant rank zero, and the notion of a local null direction is lost.\par
The Carroll foliation is, in general, singular, i.e., the dimension of the leaves is not constant (see \cite{Laurent-Gengoux:2024} for further details of singular foliations). The leaves of the Carroll foliation we will refer to as \emph{Carroll leaves} and the leaf containing the point $m \in M$ we denote as $\cO_m$. Recall that any given point lies in exactly one leaf.  A theorem of Stefan--Sussmann tells us that the set of Carroll leaves $\{ \cO_m~~|~~ m \in M\}$ forms a partition of $M$.  We have five types of leaves in a general Carroll foliation:
\begin{enumerate}
\item $1$-dimensional leaves diffeomorphic to $\R$;
\item $1$-dimensional leaves diffeomorphic to $S^1$;
\item $1$-dimensional leaves diffeomorphic to a closed interval - these are bounded by singular points;
\item $1$-dimensional leaves diffeomorphic to a half-open interval - these are half bounded by a singular point;
\item $0$-dimensional leaves - these are the singular points.
\end{enumerate}
\begin{definition}\label{def:LPath}
Let $(A, [-,-], \rho, g, L)$ be a Carrollian Lie algebroid. A \emph{$L$-path} is a smooth path $\alpha : I \rightarrow A$, over the base path $\gamma(t) := \pi(\alpha(t))$ ($I = [t_0, t_1] \subset \R$) such that 
\begin{enumerate}
\item $\rho(\alpha(t)) = \dot{\gamma}(t)$ for all $t \in I$; and
\item $\Image(\alpha)\subset L $.
\end{enumerate}
\end{definition}
\begin{proposition}\label{prop:LSingLeaf}
 Let $\alpha : I \rightarrow A$ be a $L$-path with base path $\gamma : I \rightarrow M$ on a Carrollian Lie algebroid $(A, [-,-], \rho, g, L)$. Then $\Image(\gamma)$ lies in a single leaf of the associated Carroll foliation. 
\end{proposition}
\begin{proof}
From the definition of an $L$-path (Definition \ref{def:LPath}) it is clear that $\dot{\gamma}(t) \in \rho|_L (L_{\gamma(t)}) \subset \sT_{\gamma(t)}M$ for all $t \in I$. Thus, $\gamma$ is an integral curve of the Carroll distribution $\mathcal{C} := \rho(L)$. It is known that any integral curve of an involutive distribution must be contained entirely within a single leaf of the associated foliation; see \cite[Proposition 1.1.6]{Laurent-Gengoux:2024}. Thus, $\Image(\gamma)$ lies in a single leaf of the associated Carroll foliation.
\end{proof}
We interpret $L$-paths $\alpha : I \rightarrow A$ as \emph{admissible generalised worldlines} of massive particles. In particular, their generalised velocity is confined to lie in the kernel of the metric. The \emph{Carroll time} of a particle is understood as a parameter $t \in I$ describing the $L$-path. The \emph{physical worldline} of a massive particle is  understood as the base path $\gamma(t) := \pi(\alpha(t)) \subset M$, which by Proposition \ref{prop:LSingLeaf}, lies in a single Carroll leaf for all $t \in I$;  this we understand as defining Carroll causality.  At non-singular points, we can view a particle as being `stationary' in that it only moves in the `null direction' defined by the Carroll leaf. Thus, no matter the speed of the path taken, the massive particle can not move transversely to the leaf. This means that the ``absolute future'' and ``absolute past'' are reduced to a single spatial point.  However, a particle may be on a Carroll leaf that reaches a singular point, and the `null direction' dissolves and the particle becomes `frozen' at the singular point; this is clear as at a singular point $\dot{\gamma}$ is the zero vector.  
\begin{definition}
A Carrollian Lie algebroid is said to be an \emph{$L$-regular Carrollian Lie algebroid} if $\ker(\rho|_L)$ is trivial, i.e., the Carrollian distribution $\mathcal{C} \subset \sT M$ is regular (non-singular) and of rank $1$.
\end{definition}
Note that for an $L$-regular Carrollian Lie algebroid, the Carroll distribution $\mathcal{C}$ defines an integrable foliation of the base manifold $M$. However, in general, this foliation is not regular or simple; the leaves of the foliation are not necessarily diffeomorphic to $\R$ and might exhibit more complex behaviour, such as being dense or closing on themselves. In particular, it is not the case, in general, that $M$ is a fibre bundle with typical fibre $\R$. \par
Recall that a foliated atlas is only a well-defined notion for a regular (non-singular) foliation, i.e., the associated distribution must be of constant (non-zero) rank. Thus, $M$ admits a foliated atlas if and only if the Carrollian Lie algebroid is $L$-regular.  However, the classical Frobenius theorem tells us that in the neighbourhood of any non-singular point of the Carroll foliation, there always exist coordinates adapted to the foliation. That is, we can always find coordinates $(x^a , t)$ about a non-singular point. Here, the coordinate $t$ describes the `null direction', and the coordinates $x^a$ describe the `traversal directions'. Moreover, admissible changes of coordinates must locally respect the Carroll foliation, and so are of the form
\begin{equation}\label{eqn:AdmCorChange}
    x^{a'} = x^{a'}(x)\,, \qquad t' = t'(x,t)\,.
\end{equation}
The above coordinate changes \eqref{eqn:AdmCorChange} are  interpreted as \emph{Carroll diffeomorphisms} (see \cite{Ciambelli:2018}), i.e., local diffeomorphisms that preserve the Carrollian foliation in the neighbourhood of a non-singular point.   We stress that while $M$ is a manifold, and so admits an atlas by definition, the notion of foliated coordinates is ill-defined around a singular point of the Carroll foliation. Moreover, the generator of the null direction is, in adapted coordinates, the local vector field $\partial_t$, which is precisely the form expected from standard Carrollian geometry. Note that, under coordinate changes \eqref{eqn:AdmCorChange}, 
\begin{equation}
\frac{\partial}{\partial t} \longmapsto \frac{\partial}{\partial t'} = \left(\frac{\partial t}{\partial t'}\right)\frac{\partial}{\partial t}\,,
\end{equation}
and so the local form of Carroll vector field is not invariant, but the null direction it defines \emph{is} invariant, and so a well-defined geometric concept.
\begin{proposition}
If $(A, [-,-], \rho, g, L)$ is an $L$-regular Carrollian Lie algebroid such that the underlying Lie algebroid is transitive and $\Rank(A) = \dim(M)$, then there is an isomorphism of vector bundles $E := A /L \stackrel{\sim}{\rightarrow} \sT M/\mathcal{C}$. 
\end{proposition}
\begin{proof}
First note that  $\sT M/\mathcal{C}$ is a vector bundle  only when $\mathcal{C}$ is non-singular, meaning that the $L$-regularity condition is essential. From the definition of the Carroll distribution, $\mathcal{C}:= A/L$, we observe that for any section $u \in \Sec(A)$, the image $\rho_u$ is a well-defined section of $\sT M/\mathcal{C}$. Thus, we have an induced anchor-like map
$$\bar{\rho} : E \rightarrow \sT M/\mathcal{C}\,,$$
which is defined by taking any section $\bar{u} \in \Sec(E)$, lifting it to a section $u \in \Sec(A)$, and defining $\bar{\rho}(\bar{u}) := \tau(\rho_u)$, where $\tau: \sT M \rightarrow \sT M/\mathcal{C}$. The claim is that the map $\bar{\rho}$ is an isomorphism.\par 
First, we establish that the ranks of the vector bundles are equal. Due to the fact that both $L$ and $\mathcal{C}$  are line bundles, it is clear that $\Rank(E) = \Rank(A) -1$, and $\Rank(\sT M/\mathcal{C}) = \dim(M) -1$. Then via assumption, $\Rank(A) = \dim(M)$ implies that $\Rank(E) = \Rank(\sT M/\mathcal{C})$. As the ranks of the vector bundles are indeed equal, it is sufficient to show that $\bar{\rho}$ is a surjection.  Note that by the conditions of the proposition $\rho: A \rightarrow \sT M$ is a surjection. Furthermore, the projection $\tau: \sT M \rightarrow \sT M/\mathcal{C}$ is also a surjection. The composition of two surjections is again a surjection, and thus $\bar{\rho} : E \rightarrow \sT M/\mathcal{C}$ is a surjection.  Thus, as we have a surjection between vector bundles of equal rank, we conclude that $\bar{\rho}$ is an isomorphism of vector bundles.
\end{proof}
\begin{corollary}
If $(A, [-,-], \rho, g, L)$ is an $L$-regular Carrollian Lie algebroid such that the underlying Lie algebroid is transitive and $\Rank(A) = \dim(M)$, then the vector bundle $\sT M/\mathcal{C}\rightarrow M$ comes equipped with an induced non-degenerate metric given by 
$$g_{\sT M/\mathcal{C}}(\bar{X}, \bar{Y}):= g_E(\bar{\rho}^{\,-1}(\bar{X}), \bar{\rho}^{\,-1}(\bar{Y}))\,,$$
with $\bar{X}, \bar{Y} \in \Sec(\sT M/\mathcal{C})$.
\end{corollary}
The metric $g_{\sT M/\mathcal{C}}$ is viewed as a ``partial metric'' on $M$ in that we have a non-degenerate metric on the transverse directions to the Carrollian distribution. This mimics the classical situation on Carrollian manifolds where there is an underlying  metric transverse to the `null direction’. We stress that the construction of this non-degenerate metric is only well-defined for the case of  transitive $L$-regular Carrollian Lie algebroids with $\Rank(A) = \dim(M)$; such Carrollian Lie algebroids we view as the  minimal generalisation of (weak) Carrollian manifolds. However, in order to deal with singular Carrollian distributions, the $L$-regularity needs to be dropped.
%
%%%%%%%%%%%%%%%%%%%%%%%
%
\subsection{Specific classes of Carrollian Lie algebroids}\label{subsec:Examples}
In this subsection, we present some foundational examples of Carrollian Lie algebroids. We remark that the underlying Lie algebroid structures are well-known mathematical examples.
\begin{example}[Carrollian Tangent Algebroids]
A weak Carrollian manifold $(M, g, \kappa)$ can be viewed as a framed Carrollian Lie algebroid  by setting $A:=\sT M$ with its standard Lie algebroid structure and defining $L := \ker(g)$ with $\kappa$ defining the fixed frame.  In this case, as the anchor is the identity map, we have $\mathcal{C} = L$ and thus the Carroll distribution is non-singular. 
\end{example}
\begin{example}\label{exm:VectorField}
As a simple example to demonstrate that the singular nature of a Carrollian Lie algebroid is in the anchor, we consider the following. Let $M$ be a smooth manifold and $X \in \Vect(M)$ be a chosen vector field. The trivial line bundle $A := M \times \R$ can be considered as a Lie algebroid as follows. First note that $\Sec(A) \simeq C^\infty(M)$. The Lie bracket is defined as 
$$[\psi, \chi] :=  -X(\psi)\chi + \psi X(\chi)\,,$$
for all $\psi , \chi \in \Sec(A)$.  The Jacobi identity  for the above bracket can be directly verified, and the anchor map shown to be $\rho_\psi(-) := \psi X(-)$.  We can then equip  $A$ with the zero metric, i.e., $g(\psi,\chi) =0$ for all sections of $A$.  Clearly, we have $L = A$, and thus we have a Carrollian Lie algebroid.  The Carroll distribution $\mathcal{C} := \rho(L) = \rho(A)$ is generated by the chosen vector field $X$. If this vector field is singular, then we have a singular  Stefan--Sussmann distribution.\par 
As a specific example, consider $M = \R^2$, which we equip with standard rectangular coordinates $(x,y)$, and the vector field $X = y \partial_x$. The singular locus of $X$ is the whole $x$-axis, i.e, the points defined by $(x,0)$. Thus, we have a  singular Stefan--Sussmann distribution.   In this case, the associated Carroll foliation consists of three distinct kinds of leaves. 
\begin{itemize}
    \item Upper Leaves: for every $c >0$, $\mathrm{L}^+_c := \{(x,c)~~|~~ x \in \R \}$\,;
    \item Lower Leaves: for every $c <0$, $\mathrm{L}^-_c := \{(x,c)~~|~~ x \in \R \}$\,;
    \item Singular Leaves: for every $x_0 \in \R$, $\mathrm{L}_{x_0}:= \{ (x_0, 0)\}$\,.
\end{itemize}
This foliation (in fact, a stratification) means that Carroll particles may move on either the upper or lower leaves, but become trapped or frozen at the $x$-axis and cannot cross. 
\end{example}
\begin{remark}
Observe that $M \simeq M \times \{ 0\}$ is a rank-$0$ vector bundle, and so cannot support a Carrollian structure; there are no line subbundles.  If the vector bundle $A$ is rank one, then via the above example, the Carrollian structure is trivial, i.e., $A = L$ and the degenerate metric is the zero metric. Thus, we require $\textrm{rank}(A) \geq 2$ for the possibility of non-trivial Carrollian structures.
\end{remark}
\begin{example}[Carrollian Lie Algebras]
A Carrollian Lie algebroid over a point, which we refer to as a \emph{Carrollian Lie algebra}, consists of 
\begin{enumerate}[itemsep=0.5em]
\item a real Lie algebra $(\mathfrak{g}, [-,-])$;
\item a one-dimensional Lie subalgebra $\mathfrak{l} \subset \mathfrak{g}$; and
\item a degenerate symmetric bilinear form $g : \mathfrak{g} \times \mathfrak{g} \rightarrow \R$ such that $\ker(g) =\mathfrak{l}$.
\end{enumerate}
Note that as the Lie subalgebra $\mathfrak{l}$ is one-dimensional, it is an abelian Lie algebra, i.e., the Lie bracket is zero.   As the anchor is the zero map, the Carroll foliation is identified with the single point, and so has `totally collapsed'. Dynamics here are trivial; a Carroll particle is frozen at the single point.\par 
Specifically, non-semisimple Lie algebras $\mathfrak{g}$ where the Killing form $k =:g$ (which is a symmetric bilinear form) has a rank-1 kernel are examples. The metric is defined as $g(u,v) := \tr \big( \ad_u \circ \ad_v\big)$; recall $\ad_u v := [u,v]$. In this case, the structure is automatically  ad-invariant, i.e.,
$$g([u,v], w) + g(v,[u,w]) =0\,,$$
for all $u,v,w \in \mathfrak{g}$. The Levi decomposition (see, for example, \cite[Chapter 3]{Varadarajan:1984}) leads to the identification $\mathfrak{g} \cong \mathfrak{s} \ltimes \mathfrak{l}$, where $\mathfrak{s}$ is a semisimple Lie algebra. Such Lie algebras we refer to as \emph{ad-invariant Carrollian Lie algebras}.
\end{example}
\begin{remark}
Figueroa-O’Farrill \cite{Figueroa-O’Farrill:2023} defines a slightly more restrictive notion where the kernel must be spanned by a central element, i.e., the spanning element commutes with all the elements of $\mathfrak{g}$. In this case, the semidirect product reduced to a direct product, i.e., $\mathfrak{g} \cong \mathfrak{s} \times \mathfrak{l}$. Thus, Figueroa-O’Farrill's ad-invariant Carrollian Lie algebras provide an important class of Carrollian Lie algebroids.
\end{remark}
\begin{example}[Carrollian Action Lie Algebroids]
An action of a Lie algebra $\mathfrak{g}$ on a manifold $M$ is a Lie algebra homomorphism $\mathsf{a} : \mathfrak{g} \rightarrow \Vect(M)$. The associated action Lie algebroid is defined as follows. The vector bundle structure is given by $A := M \times \mathfrak{g}$ (over $M$). Thus, sections are $\Sec(A) \simeq C^\infty(M)\otimes \mathfrak{g}$. Basic sections are of the form $u = f \,\zx$, with $f \in C^\infty(M)$ and $\zx \in \mathfrak{g}$ is an element of a basis of the Lie algebra. General sections can be written as a sum of basic elements with smooth functions as the coefficients. The anchor is defined as $\rho_{f\, \zx} := f\, \mathsf{a}_\zx \in \Vect(M)$. The Lie bracket is defined via basic sections as 
$$[f_1\, \zx_1 , f_2 \, \zx_2] := f_1 \mathsf{a}_{\zx_1}(f_2) \, \zx_2 -f_2 \mathsf{a}_{\zx_2}(f_1) \, \zx_1 + f_1 f_2 \, [\zx_1, \zx_2]\,.$$
We now equip this action Lie algebroid with a Carrollian structure:
\begin{enumerate}[itemsep=0.5em]
\item A degenerate metric $g$ on $M \times \mathfrak{g}$;
\item A trivial line bundle $L \subset A$ such that $\ker(g) = L$.
\end{enumerate}
As the line bundle $L$ is trivial, there exists a nowhere vanishing section $\sigma : M \rightarrow M \times \mathfrak{g}$ that provides a frame for $\Sec(L)$. The Carrollian distribution $\mathcal{C} := \rho(L)$, is then spanned by $\rho_\sigma = \mathsf{a}_\sigma \in \Vect(M)$. Note that a Lie algebra action may have a non-trivial kernel, and so the Carroll distribution may be singular. In particular, it is well-known that a Lie algebra action is injective (so has a trivial kernel) if and only if the associated group action is free.  Thus, in general, one expects singular Carroll distributions for infinitesimal actions associated with non-free group actions.\par 
 As a specific example, consider the standard action of $\mathfrak{gl}_2(\R)$ on $\R^2$. This corresponds to a non-free group action of $GL_2(\R)$ as the zero vector in $\R^2$ is a fixed point.  A section of $A = \R^2 \times \mathfrak{gl}_2(\R) $ can be written as
$$\zx = \begin{pmatrix}
    \alpha & \beta\\
    \gamma & \delta 
\end{pmatrix}\,, $$
where the components are functions on $\R^2$, and the action can be expressed as 
$$\rho_\zx = (\alpha\,x + \beta \, y) \frac{\partial}{\partial x} + (\gamma\,x + \delta\, y) \frac{\partial}{\partial y}\,, $$
where we have used rectangular coordinates $(x,y)$ on $\R^2$. We equip this Lie algebroid with the metric $g(\zx_1, \zx_2) := \beta_1 \beta_2 + \gamma_1 \gamma_2 + \delta_1 \delta_2$. Clearly, $\ker(g)$ is spanned by
$$\sigma = \begin{pmatrix}
    1 & 0\\
    0 & 0 
\end{pmatrix}\,,$$
and so we have a Carrollian Lie algebroid.  The Carroll distribution $\mathcal{C}$ is thus generated by $\rho_\sigma =  x \partial_x$. The singular locus of this vector field is the $y$-axis, i.e., all the points $(0,y)$. The associated Carroll foliation is identical to that found in Example \ref{exm:VectorField} upon exchanging the $x$-axis and $y$-axis.
\end{example}
\begin{example}[Riemannian Lie Algebroids + a Vector Field]\label{exa:RieLieAlgVec}  Let $M$ be a smooth manifold and $X \in \Vect(M)$ be a chosen vector field (not necessarily non-singular). The trivial line bundle $L := M \times \R$ can be considered as a Lie algebroid as follows. First note that $\Sec(L) \simeq C^\infty(M)$. The Lie bracket is defined as 
$$[\psi, \chi] :=  -X(\psi)\chi + \psi X(\chi)\,,$$
for all $\psi , \phi \in \Sec(L)$.  The Jacobi identity  for the above bracket can be directly verified, and the anchor map shown to be $\rho_\psi := \psi X$.\par 
Let $(A_0, [-,-], \rho, \hat{g})$ be a Riemannian Lie algebroid over $M$. The direct sum $A := A_0 \oplus L$ comes with a Lie algebroid structure as follows. Sections of $A$  are of the form $u+\psi$, with $u \in \Sec(A)$ and $\psi \in \Sec(L)$. The Lie bracket and anchor are defined as
$$[u + \psi , v + \chi]:= [u,v] - X(\psi)\chi + \psi X(\chi)\,, \qquad \rho_{u + \psi}(f):= \rho_u(f) +\psi X(f)\,.$$
Note that the cross-terms are zero,  i.e., $[u,\chi] =[v, \psi] =0$. The induced degenerate metric on $A$ is given by $g(u+ \phi, v + \chi) := \hat{g}(u,v)$. Clearly, $\ker(g) = \Sec(L)$. Moreover, the Carroll distribution and associated Carroll foliation  will be singular if the vector field $X$  is singular (see Example \ref{exm:VectorField}).
\end{example}
\begin{remark}
The above example directly generalises to the direct sum of a Riemannian or pseudo-Riemannian Lie algebroid and any Lie algebroid (over the same base) of rank-1 that is trivial as a line bundle.    
\end{remark}
\begin{example}[Spacetimes] As a specification of the previous example, consider a spacetime $(M, \hat{g} , \tau)$, that is a $d+1$-dimensional Lorentzian manifold equipped with a nowhere vanishing timelike (smooth) vector field $\tau \in \Vect(M)$ that defines the future direction at every point. We can associate a canonical Carrollian Lie algebroid with any spacetime, where $A = \sT M \oplus L$ and the product structure is defined in Example \ref{exa:RieLieAlgVec} using $\tau$ and the standard Lie algebroid structure on the tangent bundle. The Carroll distribution in this case is regular. 
\end{example}
\begin{example}[Carrollian Atiyah Algebroids]
Consider a principal $G$-bundle $\tau : P \rightarrow M$, where we denote the right action as $\mathsf{a}: P \times G \rightarrow P$. For any $\mathsf{g} \in G$, we have a diffeomorphism $\mathsf{a}_{\mathsf{g}} : P \rightarrow P$. A Carrollian structure $(\hat{g}, \kappa)$ on $P$ is said to be a $G$-invariant Carrollian structure if $\hat{g}(\zx, \eta) = \hat{g}\big((\mathsf{a}_{\mathsf{g}})_* \zx,(\mathsf{a}_{\mathsf{g}})_*\eta \big)$, and $(\mathsf{a}_\mathsf{g})* \kappa = \kappa$, for all $\zx, \eta \in \Vect(P)$ and $\mathsf{g} \in G$.\par 
The Atiyah algebroid is then defined as $A := \sT P/M \stackrel{\pi}{\rightarrow} M$, with the anchor $\rho : A \rightarrow \sT M$ being the restriction of $\rmd \tau : \sT P \rightarrow \sT M$ (which is $G$-invariant), and the Lie bracket is the restriction of the standard Lie bracket of vector fields, Note that sections of $A$ are $G$-invariant vector fields and $\ker(\rho) \simeq P \times_G \mathfrak{g}$, where $\mathfrak{g}$ is the Lie algebra of $G$.\par 
We have the natural inclusion $\iota : \Sec(A) \hookrightarrow \Vect(P)$ and so we define 
$$g(u,v) := \hat{g}(\iota(u), \iota(v))\,,$$
for arbitrary $u,v \in \Sec(A)$. As $\hat{g}$ is $G$-invariant, $\hat{g}(\iota(u), \iota(v))$ is constant along the fibres of $P$ (recall that the action of $G$ on $P$ preserves the fibres). Thus, $g(u,v)$ is a well-defined smooth function on $M$ and so defines a degenerate metric on $A$. As the Carroll vector field $\kappa$ is $G$-invariant, we define $\sigma := \iota^{-1} \kappa \in \Sec(A)$, which is a nowhere vanishing section of $A$. Thus, we have the trivial line subbundle $L \subset A$ defined via $\Sec(L) := \Span\big\{  \sigma \big\}$.\par 
As $\kappa$ generates the kernel of $\hat{g}$, it is clear that $g(\sigma, u) = \hat{g}(\iota(\sigma), \iota(u)) = \hat{g}(\kappa, \iota(u))=0$, for all $u \in \Sec(A)$. Thus, $\ker(g) \subset \Sec(L)$. As the initial structure is $G$-invariant and $\ker(\hat{g})$ is rank one,  the reduced metric $g$ on $A$ also has a rank-1 kernel. Since $L$ is a one-dimensional subbundle contained in $\ker(g)$, we observe that $\Sec(L) =\ker(g)$. \par 
Thus, $(\sT P/G, [-,-], \rho, g, L)$ is a  Carrollian Lie algebroid. It is entirely possible, in general, that there are  points $m \in M$, such that $\sigma_p$ is a vertical vector and so $\rho(\sigma_p) =0$. Meaning that, in general, the Carroll distribution $\mathcal{C} := \rho(L)$  will be a singular Stefan--Sussmann distribution.
\end{example}
\begin{remark}
We remind the reader that while every Lie groupoid can be differentiated to construct a Lie algebroid, not every Lie algebroid can be integrated to a Lie groupoid; there are obstructions to the integration of Lie algebroids (see \cite{Crainic:2003}). The integrating objects for Carrollian Lie algebroids, which we refer to as \emph{Carrollian Lie groupoids}. That is, a Lie groupoid $\mathcal{G}$ that is the integrating object for $A$, equipped with 
\begin{enumerate}[itemsep=0.5em]
\item a right or left invariant degenerate metric $G$, depending on whether $A$ is defined in terms of right or left invariant vector fields;
\item a one-dimensional Lie subgroupoid $\Gamma \subset \mathcal{G}$ that is the integrating object for $L$, such that  $\ker(G) = \sT \Gamma$.
\end{enumerate}
 This is consistent with the literature on the integrating objects for Riemannian Lie algebroids (see Boucetta \cite{Boucetta:2011}), where the resulting groupoids carry a non-degenerate  right/left invariant metric. The existence of $\Gamma \subset \mathcal{G}$ is guaranteed by the integrability of the subalgebroid $L \subset A$. We note that the resulting Lie subgroupoid $\Gamma$ may be an immersed rather than an embedded submanifold. Details of the integration problem for Carrollian Lie algebroids remain future work. We further comment that Figueroa-O’Farrill \cite{Figueroa-O’Farrill:2023} has defined and studied  bi-invariant Carrollian structures on Lie groups. The Carrollian Lie groupoids are a more general notion that includes Lie groups equipped with bi-invariant Carrollian structures.  In particular, Figueroa-O’Farrill's Carrollian Lie groups describe homogeneous spaces, while Carrollian Lie groupoids may describe infinitesimal Carrollian symmetry on more general manifolds where the symmetries only exist locally or along specific foliations. We will not further discuss Carrollian Lie groupoids in this paper.
\end{remark}

%
%%%%%%%%%%%%%%%%%%%%%%%
%
\subsection{Mixed null-spacelike hypersurfaces}\label{subsec:MixHypSur} In general, mixed character hypersurfaces,  i.e., hypersurfaces where the spacelike, timelike, or null nature can change pointwise, are complicated submanifolds of a spacetime, see Mars \& Senovilla \cite{Mars:1993} for example. While such objects seem exotic, they do appear in general relativity as, for example, thin shells.  Another example G\"odel's universe, where embedded hypersurfaces (without boundary) cannot be spacelike everywhere. In this subsection, we will build, under some simplifying assumptions, a Carrollian Lie algebroid from a mixed null-spacelike hypersurface. \par 
Consider a hypersurface (i.e., a properly embedded 3-dimensional submanifold) $\Sigma \subset M$ of a 4-dimensional Lorentzian manifold $(M, g_M)$ that is of mixed null-spacelike character. The hypersurface decomposes into components $ \Sigma = \Sigma_{null} \cup \Sigma_{space}$ such that: 
\begin{itemize}
\item the normal vector of  $\Sigma_{null}$ is null, and the induced metric $g_M|_{\Sigma_{null}}$ is degenerate and its pointwise diagonalisation is $\diag(1,1,0)$; and
\item the normal vector of $\Sigma_{space}$ is timelike, and the induced metric $g_M|_{\Sigma_{space}}$ is non-degenerate and its pointwise diagonalisation is $\diag(1,1,1)$.
\end{itemize}
We will make two simplifying geometric assumptions:
\begin{enumerate}
    \item  $\Sigma_{null}$ and $\Sigma_{space}$ are submanifolds of $\Sigma$; however, they need not be connected submanifolds;
    \item  there exists a \emph{characteristic vector field} $\kappa \in \Vect(\Sigma)$, i.e., $\kappa|_{\Sigma_{space}}$ identically vanishes,  and $\kappa|_{\Sigma_{null}}$ is nowhere vanishing on $\Sigma_{null}$.
\end{enumerate}
 We stress that the vector field $\kappa$ is not the normal vector field of the hypersurface, rather it captures the null-direction or lack thereof. 
It must be noted that these assumptions are very restrictive and rare for realistic spacetime embeddings.  However, we will make these assumptions in order to avoid technical issues. In particular, a globally defined $\kappa$ may need to be relaxed, and the constructions only applied locally in more physically reasonable embeddings.
 \par 
The tangent bundle restricted to $\Sigma_{null}$ defines a short filtration of vector bundles over  $\Sigma_{null}$
\begin{equation}\label{eqn:HypSurFil}
 L_{null} \subset L_{null}^\perp \subset \sT \Sigma|_{\Sigma_{null}}\,,
\end{equation}
where $L_{null}$ is the null line bundle and $L_{null}^\perp$ is its orthogonal complement in $\sT \Sigma|_{\Sigma_{null}}$. Note, we have a vector bundle over the null component $\Sigma_{null}$. The quotient bundle $ S := L_{null}^\perp/ L_{null}$ is the rank-2 screen bundle. We remark that we do not have such a filtration of vector bundles on $\Sigma_{space}$ as the only vector that is null and spacelike is the zero vector. Moreover $\kappa|_{\Sigma_{null}}$ provides a global frame for $L_{null}$.\par 
The core idea is to replace the short filtration \eqref{eqn:HypSurFil} with a direct sum of vector bundles over $\Sigma$ that separates the null direction and its transversal directions and use this to construct a Carrollian Lie algebroid over $\Sigma$ (see Example \ref{exa:RieLieAlgVec}). %
 \begin{enumerate}
\item As $\Sigma_{null}$ admits a null generator by assumption, specifically $\kappa|_{\Sigma_{null}}$, $L_{null}$ is  a trivial line bundle over $\Sigma_{null}$. We then define $L \rightarrow \Sigma$ as the smooth trivial extension of  $L_{null}$. That is, we have a trivial line bundle $L \simeq \Sigma \times \R$, such that $L|_{\Sigma_{null}} \simeq L_{null}$ as line bundles. Note that all trivial line bundles over a given base are isomorphic, thus $L$ is guaranteed to exist and is unique up to isomorphism. We can then consider the characteristic vector field as a section of $L$, i.e., $\kappa \in \Sec(L)$. The trivial line bundle $L$ is then a Lie algebroid with the structure being given by 
$$[\psi, \chi] :=  -\kappa(\psi)\chi + \psi \kappa(\chi)\,, \qquad \rho_\psi(f) :=  \psi \,\kappa(f)\,,$$
for all $\psi , \phi \in \Sec(L)$ and $f \in C^\infty(\Sigma)$.
\item The screen bundle $S := L_{null}^\perp/L_{null}$ is a rank-2 vector bundle over $\Sigma_{null}$. Moreover, $S$ is equipped with a positive-definite metric defined as follows. First $\tau : L_{null}^\perp \rightarrow S$ is a quotient map. For any $s,t \in \Sec(S)$, and lifts $\bar{s}, \bar{t} \in L_{null}^\perp$, i.e., we have $\tau(\bar{s}) =s$ and $\tau(\bar{t}) =t$, we define
$$g_S(s,t) := g_M|_{\Sigma_{null}}(\bar{s}, \bar{t})\,.$$
It can be directly verified that the metric is independent of the lift and is positive-definite.  Note that as $g_M$ is smooth, the induced metric $g_S$ is also smooth. \par 
We then define the vector bundle $A_0$ as the smooth extension of $S$ to all of $\Sigma$. That is, we construct a rank-2 vector bundle over $\Sigma$, such that $A_0|_{\Sigma_{null}} \simeq S$. The existence of such a smooth extension is not automatic for non-trivial $S \rightarrow \Sigma_{null}$; there are topological obstructions. To avoid technical issues, we assume that $S$ is trivial. Due to our assumption that $\Sigma_{null}$ is (smoothly) embedded, the trivial bundle $A_0 \simeq \Sigma \times \R^2$ is guaranteed to exist and is unique up to isomorphism. For concreteness, we may take $\Sigma_{null} \simeq \R^2 \times \R$ as any vector bundle over a contractible base can be trivialised. Examples of such null hypersurfaces arise in the study of plane gravitational waves and thin shells in flat spacetime.  \par 
The vector bundle $A_0$ comes with an inherited non-degenerate metric $g_{A_0}$, which is a smooth extension of $g_S$; standard results in differential geometry show that such a smooth extension can always be constructed via a partition of unity, however, it is not unique. The vector bundle $A_0$ is then considered a Lie algebroid by equipping it with the zero Lie bracket and zero anchor. Note that, in general, $A_0$ does not come with a natural non-trivial Lie algebroid structure. Specifically,  while $\sT \Sigma \simeq  A_0 \oplus L$, it is not, in general, the case that $A_0$ (thought of as a distribution on $\Sigma$) is integrable. 
\end{enumerate} 
Given the above constructions, a Carrollian Lie algebroid describing the mixed hypersurface is the direct sum $A := A_0 \oplus L$ where the Lie bracket and anchor are defined as
$$[u + \psi , v + \chi]:=  -\kappa(\psi)\chi + \psi \kappa(\chi)\,, \qquad \rho_{u + \psi}(f):=  \psi \, \kappa(f)\,.$$
The induced degenerate metric on $A$ is given by $g(u+ \phi, v + \chi) := \hat{g}(u,v)$. Clearly, $\ker(g) = \Sec(L)$. The Carroll distribution $\mathcal{C} := \rho(L)$ is then the singular Stefan--Sussmann distribution generated by $\kappa$. \par 
Some remarks are in order:
\begin{itemize}
\item Carrollian Lie algebroids associated with a mixed null-spacelike hypersurface (with our stated assumptions) are not unique; we have  constructed a \emph{minimal model}. 
\item The physically meaningful part of the construction is the Carrollian distribution on $\Sigma$, and this is independent of the non-canonical extended metric $g_{A_0}$. 
\item  Via the construction above, we have $\sT \Sigma \simeq A_0 \oplus L$, and in particular $A_0$ is rank-2 with a trivial Lie algebroid structure with a non-canonical Riemannian metric. However, one could equip $A_0$ with an independent Lie algebroid structure, or consider some non-trivial Lie algebroid $A'_0$ of some arbitrary (non-zero) rank, equipped with a  non-degenerate metric.  The situation is similar to the BV--BRST formalism in field theory, where there is no unique choice of ghosts, antifields, and antighost, etc. The physics is independent of these choices, and for the case of Carrollian Lie algebroids, it is $\mathcal{C}$ that is the physically relevant aspect.
\end{itemize}
%
%%%%%%%%%%%%%%%%%%%%%%%
%
\subsection{Connections on Carrollian Lie algebroids}\label{subsec:Connections}
Connections, in the form of covariant derivatives specifically, are vital throughout physics and especially when constructing Lagrangian densities, defining geodesics, and (differential) conservation laws. When there are extra geometric structures on the manifold in question, it is natural to consider connections that are compatible with or respect these structures; exactly what is meant will depend on the context.  In the context of Carrollian Lie algebroids, one should consider connections such that under parallel transport, both the degenerate metric $g$  and Lie subalgebroid are preserved.  Lie algebroid connections always exist; this can be established via the use of a partition of unity and a modification of the standard proof for the existence of linear connections on a vector bundle. The question is one of compatibility with the Carrollian structure, and whether this affects the existence of such connections.
\begin{definition}
Let $(A, [-,-], \rho, g, L)$ be a Carrollian Lie algebroid. A Lie algebroid connection $\nabla$ on $A$ said to be
\begin{enumerate}[itemsep=0.5em]
\item \emph{$L$-compatible} if it restricts to $L$, i.e., $\nabla_u v \in \Sec(L)$ for all $u \in \Sec(A)$ and $v \in \Sec(L)$;
\item \emph{metric compatible} if $\nabla g =0$, i.e.,
$\rho_u\big( g(v,w)\big) = g(\nabla_u v, w) + g(v, \nabla_u w)$, for all $u,v,w \in \Sec(A)$;
\item \emph{Carrollian} if it is both $L$-compatible and metric compatible.
\end{enumerate}
\end{definition} 
Recall that an $A$-path is a smooth path $\alpha : I \rightarrow A$ such that 
\begin{equation}
\rho(\alpha(t)) = \dot{\gamma}(t)\,,
\end{equation}
where the base path is defined as $\gamma := \pi \circ \alpha: I \rightarrow M$, here $I = [t_0, t_1] \subset \R$.   An $A$-path can be considered as a section of $A$ as $\alpha(t) \in A_{\gamma(t)}$ defines a point in the fibre over $m = \gamma(t)$ for every $t \in I$.
\begin{definition}
Let $(A, [-,-], \rho, L)$ be a Carrollian Lie algebroid equipped with a Carrollian connection. Then an $A$-path is a Carrollian geodesic if it satisfies 
$$\nabla_{\alpha(t)} \alpha(t) =0\,, $$
for all $t \in I$.
\end{definition}
Given an initial point $m = \gamma(t_0) \in M$ and an initial generalised velocity $a = \alpha(t_0) \in A_m$, the Picard--Lindelöf theorem guarantees the local existence and uniqueness of Carrollian geodesics. \par 
There are two generic classes of Carrollian geodesics.
\begin{enumerate}[itemsep=0.5em]
\item \emph{Carrollian particles}: If $a \in L_m$, then due to $L$-compatibility, $\alpha(t)$ remains in $L$, i.e, the geodesic is an $L$-path, and $\gamma(t)$ lies in the Carroll leaf that contains $m$.  
\item \emph{Carrollian swiftons}: If $a \in A_m/L_m$, then $\alpha(t)$ is not constrained to lie in $L$ and so $\gamma(t)$ will cut across Carrollian leaves. That is, swiftons are not confined to lie on a single Carrollian leaf, i.e., they are free to move transverse to the Carroll foliation. 
\end{enumerate}
Carrollian particles, as defined here, correspond to massive particles as they are confined to move only along Carroll leaves. The term swifton is due to Ecker et al. \cite{Ecker:2024} as a Carrollian version of a tachyon, but one that lacks the pathologies of standard tachyons. In particular,  swiftons are fields/particles capable of propagating at a non-vanishing velocity; in other words, they violate Carroll causality. In the framework here, assuming no external forces, they correspond to Carrollian swiftons as defined above. 
\begin{remark}
External forces may be included as follows. 
\begin{enumerate}[itemsep=0.5em]
\item For Carroll particles, a generalised force is understood as a section $\sF\in \Sec(L)$, and the equation of motion becomes $\nabla_{\alpha(t)} \alpha(t) = \sF(\gamma(t))$. That is, the generalised force is in the direction of $L$, and thus $\gamma(t)$ will change its velocity, but will remain on a single Carroll leaf.
\item For Carroll swiftons,  a generalised force is understood as a section $\sF\in \Sec(A)$, and the equation of motion becomes $\nabla_{\alpha(t)} \alpha(t) = \sF(\gamma(t))$. In this case, the force may have components transverse to $L$, and the $\gamma(t)$ will, as before, cut across Carroll leaves.
\end{enumerate}
\end{remark}
The question of the existence of Carrollian connections is addressed below.  We remark that the general approach to the existence statements will be to modify an arbitrarily chosen connection to be $L$-compatible and/or metric compatible. These proofs parallel the proof of  the existence of metric compatible connections on a Riemannian manifold by modifying a connection using the contorsion tensor. An important aspect of the proofs is the flexibility in defining components of connections in the degenerate direction.
\subsubsection*{$L$-Compatible Connections}
We proceed to describe the fundamental properties of $L$-compatible connections and establish their existence.
\begin{lemma}\label{lem:Lcom}
Let $(A, [-,-], \rho, g, L)$ be a Carrollian Lie algebroid and let $\sigma$ be a frame of $\Sec(L)$. Then if $\sigma \in \Sec(L)$ is parallel with respect to a given Lie algebroid connection $\nabla$ on $A$, then the connection is $L$-compatible.
\end{lemma}
\begin{proof}
The section $\sigma$ is nowhere vanishing, and as $L$ is trivial, such a section always exists. Any section of $L$ can then be written as $f\, \sigma$ where $f \in C^\infty(M)$.  Assuming that $\sigma$ is parallel, i.e., $\nabla_u \sigma =0$   for all $u \in \Sec(A)$, means that 
$$\nabla_u f\,\sigma = \rho_u(f) \sigma + f \, \nabla_u \sigma =\rho_u(f) \sigma \in \Sec(L)\,. $$
Thus, under the conditions of the lemma, $\nabla$ is $L$-compatible.
\end{proof}
It is important to note that the above lemma established a sufficient condition for a connection to be $L$-compatible.  That is, there are connections that are not of this form.  However, restricting attention to such connections will be essential in establishing the existence of $L$-compatible connections.
\begin{proposition}\label{prop:ExLConn}
$L$-compatible connections exist on  any Carrollian Lie algebroid. 
\end{proposition}
\begin{proof}
Lie algebroid connections always exist, and so we select one $\nabla^0$.  We 
then define a new connection given by $\nabla_u v := \nabla^0_u v + \Gamma_A(u,v)$. Here $\Gamma_A$ is a Lie algebroid $(1,2)$-tensor. Next we select a frame $\sigma \in \Sec(L)$, and imposing the  parallel condition, $\nabla_u \sigma := \nabla^0_u \sigma + \Gamma_A(u,\sigma) =0$, implies  $\Gamma_A (u,\sigma) = - \nabla^0_u \sigma $. As $\sigma$ is non-singular, the dual one-form  $\omega \in \Sec(L^*)$ exists, that is, there is a Lie algebroid one-form  such that $\omega(\sigma) =1$.  We can then define $\Gamma_A(u,v) := \big(\nabla^0_u \sigma\big)\, \omega(v)$. Thus, the Lie algebroid connection 
$$\nabla_u v :=   \nabla^0_u v -  \big(\nabla^0_u \sigma\big)\, \omega(v)\,,$$
exists and $\sigma$ is parallel with respect to this connection. Using Lemma \ref{lem:Lcom}, we conclude that $\nabla$ is $L$-compatible.
\end{proof}
\begin{remark}
 The proof of the above proposition does not exhaust all possible $L$-compatible connections.  Once we choose a frame of $\Sec(L)$, we can always construct an $L$-compatible connection, but this does not establish that all such connections are of this form.  
\end{remark}
\begin{proposition}
Let $\nabla$ be a $L$-compatible Lie algebroid connection on a Carrollian Lie algebroid $(A, [-,-], g , L)$. Then the $C^\infty(M)$-linear operator defined by the curvature tensor (see \eqref{eqn:RLinOp})
$$R(u,v)|_L : \Sec(L) \longrightarrow \Sec(A)\,, $$
has its image in $\Sec(L)$.
\end{proposition}
\begin{proof}
As the curvature $R$ is a tensor, it is sufficient to restrict attention to a chosen frame of $\Sec(L)$, which we denote as $\sigma$. Then as $\nabla$ is taken to be $L$-compatible, we can write $\nabla_u \sigma = f_u \, \sigma$ for the appropriate function $f_u \in C^\infty(M)$ associated with the section $u \in \Sec(A)$. Then a direct calculation gives
\begin{align*}
R(u,v)\sigma & = \nabla_u\big( \nabla_v \sigma\big)  -\nabla_v\big( \nabla_u \sigma\big)   -\nabla_{[u,v]}\sigma = \nabla_u(f_v\, \sigma)- \nabla_v(f_v\,\sigma ) - f_{[u,v]}\, \sigma\\
& = \rho_u(f_v) \sigma + f_v f_u \, \sigma - \rho_v(f_u)\sigma  - f_u f_v \, \sigma - f_{[u,v]}\, \sigma  = ( \rho_u(f_v)  - \rho_v(f_u)  - f_{[u,v]}) \, \sigma \in \Sec(L)\,,
\end{align*}
for all $u, v \in \Sec(A)$.
\end{proof}
\begin{proposition}
Let $\nabla$ be a $L$-compatible Lie algebroid connection on a Carrollian Lie algebroid $(A, [-,-], g , L)$. Then the torsion of $\nabla$ satisfies $T(u,v) \in \Sec(L)$ for all $u,v \in \Sec(L)$.
\end{proposition}
\begin{proof}
As $\nabla$ is $L$-compatible, it follows that  $\nabla_u v \in \Sec(L)$ for all $u,v \in \Sec(L)$.  Furthermore, Proposition \ref{prop:LAlgebroid} tells us that $L$  is a Lie algebroid, and so $[u,v] \in \Sec(L)$ for all $u,v \in \Sec(L)$. Thus
$$T(u,v) =  \nabla_u v - \nabla_v u -[u,v] \in \Sec(L)\,,$$
for all $u,v \in \Sec(L)$.
\end{proof} 
\subsubsection*{Metric Compatible Connections}
We next address the existence of metric compatible connections on a Carrollian Lie algebroid. Due to the degeneracy of the metric, we cannot apply the Fundamental Theorem of Riemannian Lie algebroids. Moreover, the question of torsion-free metric connections will be addressed separately.  Nonetheless, we have the following proposition. 
\begin{proposition}\label{prop:ExMetricCon}
Metric compatible connections exist on any Carrollian Lie algebroid. 
\end{proposition}
\begin{proof}
Lie algebroid connections always exist on a Lie algebroid, and so we select one $\nabla^0$. The non-metricity is measured by the following Lie algebroid tensor
$$\big(\nabla^0_u g\big )(v,w) = \rho_u \big(g(v,w) \big) - g(\nabla^0_u v,w) -  g( v ,\nabla^0_u w)\,.$$
We now define a new connection given by $\nabla_u v := \nabla^0_u v + \Gamma_A(u,v)$, here $\Gamma_A$ is a Lie algebroid $(1,2)$-tensor.  Imposing the metric compatibility condition on $\nabla$ forces an algebraic constraint on $\Gamma_A$, which we will examine. Specifically,
$$\rho_u\big( g(v,w)\big) = g(\nabla^0_u v ,w) + g( \Gamma_A(u,v) ,w) +g( v, \nabla^0_u w ) +  g( v, \Gamma_A(u,w) )\,.$$ 
Using the non-metricity of $\nabla^0$, the algebraic condition 
 $$\big(\nabla^0_u g \big)(v,w) = g(\Gamma_A(u,v),w) + g( v, \Gamma_A(u,w) )\,,$$
is deduced. The degeneracy of the metric $g$ does not fully constrain the components of the Lie algebroid tensor $\Gamma_A$; there is freedom in choosing components of $\Gamma_A$ that lie in the kernel of $g$. The number of components of $\Gamma_A$ is greater than the number of independent equations. The system of linear equations is underdetermined, and so a solution can always be found (there are, in fact, an infinite number of solutions). Thus, a metric compatible Lie algebroid connection can always be constructed from an arbitrary  connection, and so the theorem is established.
 \end{proof}
\subsubsection*{Carrollian Connections}
We now address the question of Carrollian connections, i.e., connections that are both $L$-compatible and metric compatible.  To do this, we first make the following observation.
\begin{lemma}\label{lem:MetImpL}
Let $\nabla$ be a metric compatible connection on a Carrollian Lie algebroid $(A, [-,-], \rho, L)$. Then $\nabla$ is also $L$-compatible, and so is a Carrollian connection.
\end{lemma}
\begin{proof}
The metric compatibility condition  $\rho_u\big( g(v,w)\big) = g(\nabla_u v, w) + g(v, \nabla_u w)$ implies that if $w \in \Sec(L)$, then  $g(v, \nabla_u w) =0$ for all $u,v \in \Sec(A)$. Thus, $\nabla_u w \in \ker(g) = \Sec(L)$ and so $\nabla$ is $L$-compatible.
\end{proof}
\begin{remark}
While metric compatibility implies $L$-compatibility, the converse is not true.
\end{remark}
\begin{theorem} \label{trm:CarrConExist}
Carrollian connections exist on any Carrollian Lie algebroid.
\end{theorem}
\begin{proof}
The theorem follows directly from Lemma \ref{lem:MetImpL} and the existence of metric compatible connections established in Proposition \ref{prop:ExMetricCon}.
\end{proof}
\begin{remark}
Theorem \ref{trm:CarrConExist} is the direct analogue of the Fundamental Theorem of Riemannian Lie algebroids. However, due to the degeneracy of the metric, unlike Levi-Civita connections, Carrollian connections are not unique.
\end{remark} 
For framed Carrollian Lie algebroids (see Definition \ref{def:StatFra}), it is natural to insist that the fixed frame $\sigma \in \Sec(L)$ is parallel with respect to a Carrollian connection; this is not automatic from the definition.  The question of the existence of such connections is analogous to the question of finding a metric connection on a Riemannian manifold that renders a (nowhere vanishing)  vector field parallel; such a connection can always be constructed by modifying the Levi-Civita connection.  The proof of Proposition \ref{prop:ExLConn} establishes that any framed Carrollian Lie algebroid can be equipped with an $L$-compatible connection such that $\nabla_u \sigma =0$ for all $u \in \Sec(A)$. We will use this in the proof of the following.
\begin{theorem}\label{trm:CarConFramed}
For any  framed Carrollian Lie algebroid, Carrollian connections such that the frame is parallel always exist.
\end{theorem}
\begin{proof}
Starting from an arbitrary Lie algebroid connection $\nabla^0$ on $A$, the proof of Proposition \ref{prop:ExLConn} shows that a $L$-compatible connection of the form
$$\nabla^1_u v :=   \nabla^0_u v -  \big(\nabla^0_u \sigma\big)\, \omega(v)\,,$$
can always be constructed using the frame $\sigma \in\Sec(L)$. Moreover, it is clear that the frame is parallel, i.e., $\nabla_u \sigma =0$ for all $u \in \Sec(A)$. The proof of Proposition \ref{prop:ExMetricCon} allows us to  amend $\nabla^1$ to obtain a metric compatible connection given by
$$\nabla_u v :=  \nabla^1_u v + \Gamma_A(u,v)\,.$$
We now need to impose the parallel condition to further constrain $\Gamma_A$.  Directly, $\nabla_u \sigma =  \nabla^1_u \sigma + \Gamma_A(u,\sigma) =0$, which implies 
$$\Gamma_A(u,\sigma)=0\,,$$
for all $u \in \Sec(A)$. We now demonstrate that this additional constraint is consistent with the metric compatibility condition. The non-metricity of $\nabla^1$,
$$(\nabla^1_u g)(v,w) = \rho_u \big( g( v,w )\big) -  g(\nabla^1_u v, w) -  g( v , \nabla^1_v w )\,,$$ 
together with the algebraic condition on $\Gamma_A$,
$$(\nabla^1_u g)(v,w) =  g(\Gamma_A(u,v), w ) +  g(v,\Gamma_A(u,w))\,,$$ 
implies the following:
$$(\nabla^1_u g)(\sigma,w) = - g( \nabla^1_u \sigma , w)  =  g(\Gamma_A(u,\sigma), w ) =0 \,,$$
as $\nabla^1_u \sigma =0$ by construction. Thus, as $g(\Gamma_A(u,\sigma), w ) =0$,  it must be the case that $\Gamma_A(u,\sigma) \in \ker(g)$ for all $u \in \Sec(A)$. The proof of Proposition \ref{prop:ExMetricCon} shows that the components of $\Gamma_A$ that lie in the kernel of the metric are not constrained by the metric compatibility. Thus, we can consistently choose $\Gamma_A(u,\sigma) =0$.  This establishes the result. 
\end{proof}
A further natural condition on a Carrollian connection on a Carrollian Lie algebroid is for it to be torsion-free, i.e., $T(u,v) =0$ for all $u,v \in \Sec(A)$. However, in general, the requirement for a connection to be metric compatible and torsion-free may be in conflict for degenerate metrics. Thus, while there are an infinite number of Carrollian connections on an arbitrary Carrollian Lie algebroid, the existence of torsion-free Carrollian connections places a strong constraint on the Carrollian Lie algebroid. The following proposition is already known in the context of Carrollian manifolds (see \cite{Ciambelli:2024}, and  \cite{Kozlov:2001} for a more general setting).
\begin{proposition}\label{prop:TorMeansStat}
If a Carrollian connection is torsion-free, then the Carrollian Lie algebroid is a stationary Carrollian Lie algebroid  (see Definition \ref{def:StatFra}).
\end{proposition}
\begin{proof}
A short calculation using the non-metricity of a connection and the definition of torsion leads to
\begin{equation}\label{eqn:LieDerCon}
(\mathcal{L}_u g)(v,w) = (\nabla_ug)(v,w) + g\big(T(u,v),w\big )  + g\big(v, T(u,w) \big) - g \big( \nabla_v u, w \big) - g \big(v, \nabla_w u\big)\,,
\end{equation}
for all $u,v,w \in \Sec(A)$. Assuming the connection is Carrollian and torsion-free means 
$$(\mathcal{L}_ug)(v,w) =  - g \big( \nabla_v u, w \big) - g \big(v, \nabla_w u\big)\,.$$
Then, taking $u \in \Sec(L)$, we observe that both $\nabla_v u$ and  $\nabla_w u$ are sections of $L$, and so in the kernel of $g$. Thus, $(\mathcal{L}_ug)(v,w) =0$, for all $v,w \in \Sec(A)$; that is, $\mathcal{L}_ug =0$, and so we have a stationary Carrollian Lie algebroid.
\end{proof}
\begin{proposition}\label{prop:StatTorFreCon}
Any stationary Carrollian Lie algebroid can be equipped with a torsion-free  Carroll connection.
\end{proposition}
\begin{proof}
Let $\nabla^0$ be an arbitrary Carroll connection on a stationary Carrollian Lie algebroid; the existence of such a connection is guaranteed by Theorem \ref{trm:CarrConExist}. We then define $\nabla_u v := \nabla^0_u v + \Gamma_A(u,v)$.  We now impose the torsion-free and Carrollian (metric compatible) conditions
\begin{subequations}
    \begin{align}
  \label{eqn:TorCon1} & \Gamma_A(u,v) - \Gamma_A(v,u) = - T^0(u,v)\,,\\
  \label{eqn:TorCon2}  & g (\Gamma_A(u,v), w) + g(v, \Gamma_A(u,w)) =0\,,
\end{align}
for all $u,v,w \in \Sec(A)$. We need to argue that the stationary condition implies that both the torsion-free and Carrollian conditions can simultaneously be solved.  Assuming $u \in \Sec(L)$, \eqref{eqn:LieDerCon} directly implies 
\begin{equation}\label{eqn:TorCon3}
g(T^0(u,v), w) + g(v, T^0(u,w))=0\,.
\end{equation}
\end{subequations}
Equation \eqref{eqn:TorCon2}  then tells us that $\Gamma_A(u,v) \in \ker(g)$ if either $u$ or $v$ are sections of $L$. Thus, due to the degeneracy of the metric, Equations \eqref{eqn:TorCon1} and \eqref{eqn:TorCon2} form an underdetermined system. However, we still have  consistency requirements on the initial torsion $T^0$.  Via Proposition \ref{prop:TorMeansStat}, the stationary condition is precisely the algebraic constraint, given by \eqref{eqn:TorCon3}, that guarantees the system of equations \eqref{eqn:TorCon1} and \eqref{eqn:TorCon2} is consistent and admits at least one solution $\Gamma_A$.
\end{proof}
Putting the Proposition \ref{prop:TorMeansStat} and Proposition \ref{prop:StatTorFreCon} together, we arrive at the following.
\begin{theorem}\label{trm:TorFreeCon}
A Carrollian Lie algebroid admits a torsion-free Carrollian connection if and only if it is stationary.
\end{theorem}
As a weak Carrollian manifold $(M, g, \kappa)$ can be interpreted as a Carrollian Lie algebroid structure on $\sT M$. A strong Carrollian manifold is a weak Carrollian manifold equipped with a Carrollian connection (we will relax the torsion-free condition). Thus, we are led to the following via Theorem \ref{trm:CarrConExist}, Theorem \ref{trm:CarConFramed}, and Theorem \ref{trm:TorFreeCon}.
\begin{corollary}
Any weak Carrollian manifold $(M, g, \kappa)$ can be equipped with a Carrollian connection. Moreover, Carrollian connections such that the Carroll vector field $\kappa$ is parallel can always be constructed. In other words, any weak Carrollian manifold can be made a strong Carrollian manifold. Furthermore, Carrollian connections with vanishing torsion only exist on Carrollian manifolds if and only if the degenerate metric is stationary.
\end{corollary}
It is important to note that Carrollian connections are not completely determined by the Carrollian structure on a Carrollian Lie algebroid; this is a general feature of metric compatible connections for degenerate metrics. Thus, in order to define geodesic motion  compatible with the Carrollian structure, a Carrollian connection needs to be posited. Furthermore, only on a stationary Carrollian Lie algebroid can one consider Carrollian connections that are torsion-free. Thus, on non-stationary Carrollian Lie algebroids, one must choose between a Carrollian connection or a torsion-free connection. 
\begin{example}
Let $(A_0 , [-,-]_{A_0}, \rho_{g_0}, g_{A_0})$ be a Riemannian Lie algebroid and let $(L, [-,-]_L, \rho_L)$ be a rank-1 Lie algebroid over the same base manifold. We take $L \rightarrow M$  to be a trivial line bundle.  The direct sum $A := A_0 \oplus L$ is a Carrollian Lie algebroid where $g(u + \psi, v + \zx) := g_{A_0}(u,v)$, and clearly $\ker(g) = \Sec(L)$. As $A_0$ is a Riemannian Lie algebroid, it comes canonically equipped with a Levi-Civita connection which we denote as $\nabla^{LC}$.  The trivial line bundle $L$ similarly comes equipped with the canonical trivial connection, which we denote as $\nabla^{L}$. We then define the direct sum of connections as 
$$\nabla^0_{(u+\psi)}(v+\zx) :=  \nabla^{LC}_u v + \nabla^L_\psi \zx\,.$$
By viewing sections of  $L$ as $\psi = 0 + \psi$, it is clear that $\nabla^0$ is a $L$-compatible connection. Directly checking the metric compatibility condition, using  $\ker(g) = \Sec(L)$, we arrive at
$$ \big( \nabla^0_{u + \psi}g\big)(v+ \zx, w+ \eta) = \rho_u \big( g_{A_0}(v,w)\big) + \rho_\psi \big( g_{A_0}(v,w)\big) - g_{A_0}(\nabla^{LC}_u v, w) -g_{A_0}(v ,\nabla^{LC}_u w)\,. $$
As $\nabla^{LC}$ is metric with respect to $g_{A_0}$, the non-metricity tensor reduces to
$$\big( \nabla^0_{u + \psi}g\big)(v+ \zx, w+ \eta) = \rho_\psi \big( g_{A_0}(v,w)\big)\,,$$
which, in general, is not zero. The proof of Proposition \ref{prop:ExMetricCon} established that we can `correct' $\nabla^0$ by adding a tensor $\Gamma_A$ and render the resulting connection metric compatible.  Carefully evaluating the algebraic constraint on $\Gamma_A$, using the symmetry and the kernel, leads to
\begin{equation}\label{eqn:AlgConSplit}
\rho_\psi\big (g_{A_0}(v,w)\big) =  2\, g_{A_0} \big( \hat{\Gamma}_A(\psi,v), w\big)\,,
\end{equation}
where $\hat{\Gamma}_A(\psi,v)$ is the component of $\Gamma_A(\psi,v)$ in $A_0$. As $g_0$ is nondegenerate, \eqref{eqn:AlgConSplit} uniquely defines a Lie algebroid $(1,2)$-tensor
$$ \Sec(A) \times \Sec(A) \ni (\bar{u}, \bar{v}) \longmapsto \hat{\Gamma}_A\big(\pi_L(\bar{u}),  \pi_{A_0}(\bar{v})\big)\in \Sec(A_0)\,,$$
were $\pi_L : \Sec(A) \mapsto \Sec(L)$ and $\pi_{A_{0}} : \Sec(A) \mapsto \Sec(A_0)$.  We then set all the components of $\Gamma_A$ that have their image in $\Sec(L)$ to be zero, as these components are unconstrained by the algebraic condition; this is a valid choice. The components that have their image in $\Sec(A_0)$ are defined uniquely by the above.  We have thus constructed a minimal Carroll connection on $A$ given by
$$\nabla := \nabla^0 + \Gamma_A\,.$$
\end{example}
%
%%%%%%%%%%%%%%%%%%%%%%%
%
\section{Concluding Remarks}\label{sec:ConRem}

Carrollian Lie algebroids  are a mathematically consistent framework to describe novel Carrollian geometries in which the Carroll vector field may be singular,  e.g.,  Carrollian  analogues of black holes  (see \cite{Ecker:2023}).  Carrollian Lie algebroids push the singular nature of the geometry into the anchor map, while the rest of the structure is geometrically well-behaved. In particular, the degenerate metric and the rank of its kernel remain well-defined. The Carroll distribution is, in general, a singular Stefan--Sussmann distribution, and captures the notion of singular Carroll vector fields.  The associated Carroll foliation describes Carroll causality, i.e., massive particles are confined to move along a fixed leaf of the foliation. In short, we have a robust foundation to work with exotic Carrollian structures. \par 
The existence of compatible connections on a Carrollian Lie algebroid has been established; we refer to these as Carrollian connections. Such connections are essential for defining parallel transport that respects the Carrollian leaves, i.e., distinguishes between the spatial and null directions. Moreover, Carrollian connections are needed for defining consistent matter couplings and Carrollian dynamics, including their Lie algebroid generalisations.  For the case of (pseudo-)Riemannian Lie algebroids, which include the standard case of (pseudo-)Riemannian manifolds, we have a unique connection that is both metric compatible and torsion-free, i.e., the Levi-Civita connection. While Carrollian Lie algebroids always admit Carrollian connections, they generically carry non-zero torsion. Specifically, we have shown that torsion-free Carrollian connections exist if and only if the Carrollian Lie algebroid is stationary, i.e., all sections of the kernel of the metric are Killing sections. \par 
Heuristically, the Killing condition on sections of $L$ ensures the degenerate metric is insensitive to the `null direction'; we effectively have a non-degenerate metric in the transversal directions. This implies the existence of a transversal Levi-Civita connection. The metric does not ``feel'' the components of the full connection in the null directions, allowing the freedom to choose these components to be torsion-free. \par 
Carrollian connections are far from unique. Thus, Carrollian connections, in the intrinsic picture of Carrollian geometry, need to be stated as extra geometric data. A corollary of this existence result is that Carrollian connections can always be constructed on any weak Carrollian manifold; a result undoubtedly known to experts but, to the author’s knowledge, not explicitly stated in the literature before (see Caimbelli \cite{Ciambelli:2024}, Ciambelli \&  Jai-akson \cite{Ciambelli:2025}, and  Vigneron et al. \cite{Vigneron:2025}).

\medskip
Further avenues for future exploration include:\\
\noindent \textbf{Carrollian singularities:} Further physically motivated examples of Carrollian Lie algebroids may help in developing a robust notion of a Carrollian singularity and further give insight into Carrollian gravity. For example, mixed null-space like hypersurfaces (see Subsection \ref{subsec:MixHypSur}), under some assumptions, have been shown to lead to Carrollian Lie algebroids. In particular, relaxing the condition that the screen vector bundle is trivial would cover more physical situations. However, there is a K-theoretical obstruction to the existence of a smooth extension of the screen vector bundle, and this requires careful mathematical analysis.    \\
\noindent \textbf{Galilean Lie algebroids:}  Galilean geometry is the non-relativistic limit ($c \mapsto \infty$) of Lorentzian geometry. Given the duality between Carrollian and Galilean geometry, it is expected that Galilean Lie algebroids can similarly be constructed, (see \cite{Figueroa-O’Farrill:2022,Figueroa-O’Farrill:2023}, for example).\\
\noindent \textbf{Carrollian Lie groupoids:} The global objects associated with Carrollian Lie algebroids are Carrollian Lie groupoids.  The study of the Lie groupoid case remains completely unexplored and potentially offers novel perspectives on  Carrollian Lie groups and related structures (see \cite{Figueroa-O’Farrill:2023}).  Carrollian Lie groups describe homogeneous spaces, while Carrollian Lie groupoids offer insight into  Carrollian symmetry on more general manifolds where the symmetries only exist locally. \\
\noindent \textbf{Holography and AdS/CFT:} Carrollian field theories have been studied as potential boundary theories in the context of holography (see, for example, \cite{Donnay:2022,Donnay:2023}). Speculatively, Carrollian Lie algebroids offer a framework to construct theories with singular Carroll vector fields.\\
\noindent \textbf{Carrollian-Poisson algebroids:} It is well-known that the cotangent bundle of a Poisson manifold is a Lie algebroid. This invites the study of such Lie algebroids equipped with a Carrollian structure. From a physics perspective, a classical limit of quantum gravity could well be a manifold equipped with both a pseudo-Riemannian and Poisson structure. Carrollian-Poisson algebroids fit within this philosophy as a limit of quantum gravity. The interplay of the two degenerate structures  may lead to novel results. For example, are there natural compatibility conditions, and what happens to the Carroll distribution when we have singular Poisson structures? Can Carrollian-Poisson algebroids offer a new perspective on contravariant gravity theories (see \cite{Kaneko:2017})? \\
\noindent \textbf{Condensed matter physics and hydrodynamics:} Applications of Carrollian dynamics in hydrodynamics and condensed matter physics are an emerging field (see \cite[Part III]{Bagchi:2025}). For example, fractons are an emergent quasiparticle with restricted mobility; they can only move in bound pairs. Such quasiparticles can be studied from the perspective of Carrollian physics (see \cite{Figueroa-O’Farrill:2023b}). As Carrollian Lie algebroids may appear as a reduction of invariant Carrollian structures, it is conceivable that they have applications in condensed matter physics and hydrodynamics. For instance, can the semiclassical dynamics of fractions be described using geometric mechanics on a Carrollian Atiyah algebroid? 
%
%%%%%%%%%%%%%%%%%%%%%%%
%
\section*{Acknowledgments}
The author expresses his gratitude to Daniel Grumiller for highlighting the need for singular Carroll vector fields. Appreciation is also extended to the anonymous referees for their valuable comments and suggestions.  
%
%%%%%%%%%%%%%%%%%%%%%%%
%

\end{document}